\documentclass[article]{amsart}
\usepackage{amsmath,amssymb,amsthm}
\usepackage[latin1]{inputenc}
\usepackage{version,tabularx,multicol}
\usepackage{graphicx,float, url}
\usepackage{amsfonts}

\usepackage{graphicx}
\usepackage{epsfig}
\usepackage{subfigure}
\usepackage{color}
\usepackage{float}
\usepackage{color}
\usepackage{flafter}

\vspace{0.3cm}

\theoremstyle{plain}
\newtheorem{theo}{Theorem}[section]

\newtheorem{cor}[theo]{Corollary}
\newtheorem{prop}[theo]{Proposition}

\newtheorem{lem}[theo]{Lemma}

\theoremstyle{remark}
\newtheorem{rem}[]{Remark}[section]

\newcommand{\ind}{{\bf 1}}

\begin{document}

\title{Asympotic behavior of the total length of external branches for Beta-coalescents}

\date{\today}

\author{Jean-St\'ephane Dhersin}
\address{University\'e Paris 13, Sorbonne Paris Cit\'e, LAGA, CNRS, UMR 7539, F-93430, Villetaneuse, France.}
\email{dhersin@math.univ-paris13.fr}

\author{Linglong YUAN}
\address{University\'e Paris 13, Sorbonne Paris Cit\'e, LAGA, CNRS, UMR 7539, F-93430, Villetaneuse, France.}
\email{yuan@math.univ-paris13.fr}

\begin{abstract}
{ In this paper, we consider the Beta$(2-\alpha,\alpha)$-coalescents with $1<\alpha<2$ and study the moments of external branches, in particular the total external branch length $L_{ext}^{(n)}$ of an initial sample of $n$ individuals. For this class of coalescents, it has been proved that $\displaystyle n^{\alpha-1}T^{(n)}\stackrel{(d)}{\rightarrow }T,$ where $T^{(n)}$ is the length of an external branch chosen at random,  and $T$ is a known non negative random variable. We get the asymptotic behaviour of several moments of $L_{ext}^{(n)}$. As a consequence, we obtain that for Beta$(2-\alpha,\alpha)$-coalescents with $1<\alpha<2$, $\displaystyle \lim_{n\rightarrow +\infty}n^{3\alpha-5}\mathbb{E}[(L_{ext}^{(n)}-n^{2-\alpha}\mathbb{E}[T])^2]=\frac{\left((\alpha-1)\Gamma(\alpha+1)\right)^2\Gamma(4-\alpha)}{(3-\alpha)\Gamma(4-2\alpha)}$. }
\end{abstract}

\thanks{Jean-St\'ephane Dhersin and Linglong Yuan  benefited from the support of the ``Agence Nationale de la Recherche'': ANR MANEGE (ANR-09-BLAN-0215). }

\keywords{Coalescent process, Beta-coalescent, total external branch length, Fu and Li's statistical test }

\subjclass[2010]{Primary: 60J28. Secondary: 60J25; 92D25; 60J85}
\maketitle

\section{Introduction}
\subsection{Motivation}
In a Wright-Fisher haploid population model with size $N$,  we sample $n$ individuals at present from the total population, and look backward to see the ancestral tree until we get  the most recent common ancestor (MRCA). If time is well rescaled and the size $N$ of population becomes large, then the genealogy of the sample of size $n$  converges weakly to  the Kingman $n$-coalescent (see \cite{MR671034},\cite{kingman2000oc}). During the evolution of the population, mutations may occur. We  consider the infinite sites model introduced by Kimura \cite{kimura1969nhn}. In this model, each mutation is produced at a new site which is never seen before and will never be seen in the future. The neutrality of mutations means that all mutants are equally privileged by the environment. Under the infinite sites model, to detect or reject the neutrality when the genealogy is given by the Kingman coalescent,  Fu and Li\cite{Fu1993} have proposed a  statistical test based on the total mutation numbers on the external branches and internal branches. Mutations on external branches affect only single individuals, so in practice they can be picked out  according to the model setting. In this test,  the ratio $L^{(n)}_{ext}/L^{(n)}$ between  the total external branch length  $L_{ext}^{(n)}$ and the total length $L^{(n)}$  measures in some sense the weight of mutations occurred on external branches among all. It then makes the study of these quantities relevant.

For many populations, Kingman coalescent describes the genealogy quite well. But for some others, when descendants of one individual can occupy a big ratio of the next generation with non-negligible probability, it is no more relevant. It is for example the case of some marine species (see \cite{arnason2004mitochondrial}, \cite{boom1994mitochondrial}, \cite{Eldon2006}, \cite{MR2365877}, \cite{hedgecock19942}). In this case, if time is well rescaled and the size of population becomes large, the ancestral tree converges weakly to the $\Lambda$-coalescent which is associated with a finite measure $\Lambda$ on $[0,1]$. This coalescent allows  multiple collisions. It has first been introduced by Pitman\cite{MR1742892} and Sagitov\cite{MR1742154}.  Among  $\Lambda$-coalescents, a special and important   subclass is called  Beta$(a,b)$-coalescents characterized by $\Lambda$ being a Beta distribution $Beta(a,b)$. The most popular ones are those with parameters $2-\alpha$ and $\alpha$ where $\alpha\in(0,2)$.
 
$Beta$-coalescents arise not only in the context of biology. They also have connections with supercritical Galton-Watson process (see \cite{MR1983046}), with continuous-state branching processes (see \cite{MR2120246}, \cite{berestyckismall}, \cite{2012arXiv1203.2494F}), with continuous random trees (see \cite{MR2349577}). If $\alpha=1$, we recover the Bolthausen-Sznitman coalescent which appears in the field of spin glasses (see \cite{MR1652734}, \cite{MR2309599}) and is also connected to random recursive trees (see \cite{MR2164028}). The Kingman coalescent is also obtained from the $Beta(2-\alpha,\alpha)$-coalescent by letting ${\alpha}$ tend to $2$.

{
For $Beta(2-\alpha,\alpha)$-coalescents with $1<{\alpha}<2$, a central limit theorem of the total external branch length $L_{ext}^{(n)}$ is known (see \cite{kersting2012total}). We should say that in this case, the moment method is not able to obtain the right convergence speed in the central limit theorem, which illustrates some limitations of moment calculations. }

 \subsection{Introduction and main results}
Let $\mathcal{E}$ be the set of partitions of $\mathbb{N}:=\{1,2,3,...\}$ and,  for $n\in {\mathbb N}$ , $\mathcal{E}_{n}$ be the set of partitions of $\mathbb{N}_{n}:=\{1,2,\cdots,n\}$. We denote  by $\rho^{(n)}$ the natural restriction  on $\mathcal{E}_{n}$: 
if $1\leq n\leq m\leq +\infty$ and  $\pi=\{A_{i}\}_{i\in I}$ is a partition of ${\mathbb N}_{m}$, then $\rho^{(n)}\pi$ is the partition of ${\mathbb N}_{n}$ defined by  
$\rho^{(n)}\pi=\{A_{i}\bigcap {\mathbb N}_{n}\}_{i\in I}$.  
For a finite measure ${\Lambda}$ on $[0,1]$, we denote by  $\Pi=(\Pi_{t})_{t\geq0}$ the ${\Lambda}$-coalescent process introduced independently by Pitman\cite{MR1742892} and Sagitov\cite{MR1742154}. The process $(\Pi_{t})_{t\geq0}$ is a c\`ad-l\`ag continuous time  Markovian process taking values in $\mathcal{E}$ with $\Pi_{0}=\{\{1\}, \{2\}, \{3\},...\}$. It is characterized by  the  c\`ad-l\`ag $\Lambda$ $n$-coalescent processes $(\Pi_{t}^{(n)})_{t\geq 0}:=(\mathcal{\rho}^{(n)}\Pi_{t})_{t\geq0}, n\in \mathbb{N}$. For $n\leq m\leq +\infty $, we have $(\Pi_{t}^{(n)})_{t\geq 0}=(\mathcal{\rho}^{(n)}\Pi_{t}^{(m)})_{t\geq 0}$ (where ${\Pi}^{(+\infty )}={\Pi}$). 

Let ${\nu}(dx)=x^{-2}\Lambda(dx)$.  For $2\leq a\leq b$, we set 
$$\lambda_{b,a}=\int_{0}^{1}x^{a-2}(1-x)^{b-a}\Lambda(dx)=\int_{0}^{1}x^{a}(1-x)^{b-a}\nu(dx).$$

$\Pi^{(n)}$ is  a Markovian process with values in $\mathcal{E}_{n}$, and its transition rates are given by: for ${\xi}, {\eta}\in \mathcal{E}_{n}$,  $q_{{\xi},{\eta}}={\lambda}_{b,a} $ if ${\eta}$ is obtained by merging $a$ of the  $b=|{\xi}|$  blocks of ${\xi}$ and letting the $b-a$ others unchanged, and $q_{{\xi},{\eta}}=0$ otherwise. We say that $a$ individuals (or blocks) of ${\xi}$ have been coalesced in one single individual of ${\eta}$. Remark that the process $\Pi^{(n)}$ is an exchangeable process, which means that, for any permutation $\tau$ of ${\mathbb N}_{n}$,  $\tau\circ\Pi^{(n)}\stackrel{(d)}{=}\Pi^{(n)}$.  

The process $\Pi^{(n)}$  finally reaches one  block. This final individual  is called the most recent common ancestor (MRCA). We denote by $\tau^{(n)}$ the number of collisions it takes for the $n$ individuals to be coalesced to the MRCA.

We define  by $R^{(n)}=(R^{(n)}_{t})_{t\geq 0}$  the block counting process  of  $(\Pi_{t}^{(n)})_{t\geq 0}$: $R^{(n)}_t=|\Pi_{t}^{(n)}|$, which equals the number of blocks/individuals at time $t$. Then $R^{(n)}$ is a continuous time Markovian process taking values in $\mathbb{N}_n$, decreasing from $n$ to $1$. At state $b$, for $a=2,...,b$, each of the ${b\choose a}$ groups  with $a$ individuals coalesces independently at rate $\lambda_{b,a}$. Hence, the time the process $(R^{(n)}_{t})_{t\geq 0}$ stays at state $b$ is exponential  with parameter:
\begin{align}\label{gn}
g_{b}=\sum_{a=2}^{b}{b \choose a }\lambda_{b,a}=\int_{0}^{1}(1-(1-x)^{b}-bx(1-x)^{b-1}){\nu}(dx)=b(b-1)\int_0^1t(1-t)^{b-2}\rho(t)dt,
\end{align}
where $\rho(t)=\int_t^1\nu(dx).$
We denote by $Y^{(n)}=(Y_{k}^{(n)})_{k\geq 0}$ the discrete time Markov chain associated with $R^{(n)}$.  This is a decreasing process from $Y_{0}^{(n)}=n$ which reaches 1 at the ${\tau}^{(n)}$-th jump. The probability transitions of the Markov chain $Y^{(n)}$ are given by: for $b\geq 2$, $k\geq 1$ and $1\leq l\leq b-1$,

\begin{equation}\label{pnk}
p_{b,b-l}:=\mathbb{P}(Y_{k}^{(n)}=b-l|Y_{k-1}^{(n)}=b)=\frac{{b\choose l+1}\lambda_{b,l+1}}{g_{b}},
\end{equation}

and 1 is an absorbing state.

We introduce the discrete time process $X^{(n)}_{k}:=Y_{k-1}^{(n)}-Y_{k}^{(n)}$, $k\geq 1$ with $X^{(n)}_{0}=0$. This process counts the number of blocks  we lose at the $k$-th jump.
For $i\in\left\{ 1,\ldots,n\right\}$, we define
$$T^{(n)}_i:=\inf\left\{t|\left\{ i\right\}\notin \Pi^{(n)}_t\right\}$$
as the length of the $i$-th external branch and $T^{(n)}$ the length of a randomly chosen external branch. By exchangeability, $T_i^{(n)}\stackrel{(d)}{=}T^{(n)}$. We denote by $L_{ext}^{(n)}:=\sum_{i=1}^{n}T_i^{(n)}$ the total external branch length  of $\Pi^{(n)}$, and by $L^{(n)}$ the total branch length. 

For several measures ${\Lambda}$, many asymptotic results on the external branches and their total external lengths of the ${\Lambda}$ $n$-coalescent are already known. 
\begin{enumerate}
\item If $\Lambda=\delta_{0}$, Dirac measure on $0$, $\Pi^{(n)}$ is the Kingman $n$-coalescent. Then, 
\begin{enumerate}
\item $nT^{(n)}$ converges in distribution to $T$ which is a random variable with density $f_{T}(x)=\frac{8}{(2+x)^{3}}\ind_{x\geq 0}$ (See \cite{MR2156553}, \cite{caliebe2007length}, \cite{janson2011total}). 
\item $L_{ext}^{(n)}$ converges in $L^2$ to $2$ (see \cite{Fu1993}, \cite{MR2439767}). A central limit theorem is also proved in \cite{janson2011total}.
\end{enumerate} 

\item If $\Lambda$ is the uniform probability measure on $[0,1]$, $\Pi^{(n)}$ is the  Bolthausen-Sznitman $n$-coalescent. Then $(\log n)T^{(n)}$ converges in distribution to an exponential variable with parameter 1 (see \cite{MR2554368}, \cite{yuan2013measure}). For moment results of $L^{(n)}_{ext}$, we refer to \cite{EJP2286} and for central limit theorem, we refer to \cite{kersting2013total}.
\item If $\nu_{-1}=\int_0^1x^{-1}\Lambda(dx)<+\infty$, which includes the case of the $Beta(2-\alpha,\alpha)$-coalescent with $0<\alpha<1$, then
 \begin{enumerate}
\item $T^{(n)}$ converges in distribution to an exponential variable with parameter $\nu_{-1}$ (see \cite{MR2484170,MR2684740}).
\item  $L^{(n)}/n$ converges in distribution to a random variable $L$ whose distribution coincides with that of $\int_0^{+\infty}e^{-X_t}dt$, where $X_t$ is a certain subordinator (see page $1405$ in \cite{MR2353033} and \cite{MR2256876} ), and $L_{ext}^{(n)}/L^{(n)}$ converges in probability to $1$ (see \cite{MR2684740}).
\end{enumerate}
\item If $\rho(t)=C_0t^{-\alpha}+O(t^{-\alpha+\zeta}), C_0>0, \zeta>1-1/\alpha, 1<\alpha<2$, when $t\rightarrow 0$, which includes the $Beta(2-\alpha,\alpha)$-coalescents  with $1<\alpha<2$ (with $C_0=\frac{1}{\Gamma(\alpha+1)\Gamma(2-\alpha)}$), $n^{\alpha-1}T^{(n)}$ converges in distribution to $T$ which is a random variable with density function (see\cite{dhersin2012length})
 \begin{equation}\label{density}f_{T}(x)=\frac{\alpha C_0\Gamma(2-\alpha)}{(\alpha-1)}(1+C_0\Gamma(2-\alpha)x)^{-\frac{\alpha}{\alpha-1}-1}\ind_{x\geq 0}.
\end{equation}
In the case of Beta$(2-\alpha, \alpha)$-coaelscents with $1<\alpha<2$, we refer to \cite{kersting2012total, kersting2012asymptotic} for central limit theorems of $L_{ext}^{(n)}$ and $L^{(n)}$.

\end{enumerate}

{In this paper, we consider the processes which satisfy the following assumption:
\begin{equation}\label{mainassum}
\rho(t)=C_0t^{-\alpha}+C_1t^{-\alpha+\zeta}+o(t^{-\alpha+\zeta}), C_0>0, C_1\in \mathbb{R}, \zeta>0, t\rightarrow 0.
\end{equation}
The aim is to study the moments of the total external branch length $L_{ext}^{(n)}$ of these processes.  We assume from now on that $1<\alpha<2$ and $T$ is a random variable with density (\ref{density}). Here is our main result. }

For $s\in \mathbb{R},$ we define 

\begin{equation}\label{notation}\nu^{(s)}(dx)=(1-x)^{s}\nu(dx), \quad \text{and} \quad \omega^{(s)}=\sup\{u\geq 0 ; \forall 0<y<1, \int_y^{1}(1-x)^{-u}\nu^{(s)}(dx)<+\infty\}.\end{equation}
We define in particular $\omega=\omega^{(0)}$ and notice that $ \nu^{(0)}=\nu$. The quantity $\omega^{(s)}$ gives the information on the singularity of the measure $\nu^{(s)}$ near $1$.  
\begin{theo}\label{cor1}
We assume that $\rho(t)$ satisfies (\ref{mainassum}) and that $\alpha-1<\zeta$, and $2(\alpha-1)<\omega^{(1)}.$ 
\begin{enumerate}
\item The total external branch length  $L_{ext}^{(n)}$ satisfies 
$$\lim_{n\rightarrow +\infty}n^{3\alpha-5}\mathbb{E}[(L_{ext}^{(n)}-n^{2-\alpha}\mathbb{E}[T])^2]=\frac{\Delta(\alpha)}{2},$$
where
$$\mathbb{E}[T]=\frac{\alpha-1}{C_0\Gamma(2-\alpha)}, \quad \text{and} \quad \Delta(\alpha)=\frac{\int_0^1((1-x)^{2-\alpha}-1)^2\nu(dx)}{3-\alpha}\left(\frac{\alpha-1}{C_0\Gamma(2-\alpha)}\right)^3.$$

In particular, for the Beta$(2-\alpha,\alpha)$-coalescent

$$\lim_{n\rightarrow +\infty}n^{3\alpha-5}\mathbb{E}[(L_{ext}^{(n)}-n^{2-\alpha}\mathbb{E}[T])^2]=\frac{\left((\alpha-1)\Gamma(\alpha+1)\right)^2\Gamma(4-\alpha)}{2(3-\alpha)\Gamma(4-2\alpha)},$$

with $\mathbb{E}[T]=\alpha(\alpha-1)\Gamma(\alpha).$

\item As a consequence,
$\displaystyle n^{\alpha-2}L_{ext}^{(n)}\stackrel{(L^2)}{\rightarrow }\mathbb{E}[T].$
\end{enumerate}
\end{theo}
\begin{rem}
\begin{itemize}
\item  For the second part of the theorem, the convergence in probability and almost surely can be deduced from  \cite{MR2349577}, \cite{berestycki2009recent}, \cite{berestycki2011asymptotic} by Berestycki et al for a slightly different class of $\Lambda$ coalescents.

\item The first part of the theorem gives $n^{(5-3\alpha)/2}$ as the  convergence speed for $L_{ext}^{(n)}$ tending to $n^{2-\alpha}\mathbb{E}[T]$ in the sense of second moment. But as shown in \cite{kersting2012total} for Beta$(2-\alpha, \alpha)$-coalescents,
$$\frac{L_{ext}^{(n)}-n^{2-\alpha}\mathbb{E}[T]}{n^{1/\alpha+1-\alpha}}\stackrel{(d)}{\rightarrow}\frac{\alpha(2-\alpha)(\alpha-1)^{1/\alpha+1}\Gamma(\alpha)}{\Gamma(2-\alpha)^{1/\alpha}}\zeta,$$
where $\zeta$ is a stable random variable with parameter $\alpha.$ Our moment method fails to get the right speed of convergence in distribution. 

To prove this result, the first idea is to write 
\begin{equation}\label{decom2}\mathbb{E}[(L_{ext}^{(n)}-n^{2-\alpha}\mathbb{E}[T])^2]=nVar(T_1^{(n)})+n(n-1)\mbox{Cov}(T_1^{(n)},T_2^{(n)})+(n\mathbb{E}[T_1^{(n)}]-n^{2-\alpha}\mathbb{E}[T])^2.\end{equation} 
Hence we have to get results on the moments of the external branches. This is given by the next theorems. The first one gives the asymptotic behaviour for the covariance of two  external branch lengths.
\end{itemize}
\end{rem}
\begin{theo}\label{thm1}
We assume that $\rho(t)$ satisfies (\ref{mainassum}) and that $\alpha-1<\zeta$, and $2(\alpha-1)<\omega^{(1)}.$ Then the asymptotic covariance of two external branch lengths is given by:
$$\lim_{n\rightarrow +\infty}n^{3(\alpha-1)}\mbox{Cov}(T_1^{(n)},T_2^{(n)})=\Delta(\alpha).$$
\end{theo}
\begin{rem} We will prove a more general result: This theorem is the case $(3)$ of Corollary \ref{thm4}.
Notice that $\Delta(\alpha)$ is strictly positive implies that $\mbox{Cov}(T_1^{(n)},T_2^{(n)})$ is of order $n^{3-3\alpha}$ and $T_1^{(n)}$, $T_2^{(n)}$ are positively correlated in the limit which is similar to Boltausen-Sznitman coalescent and opposite of  Kingman coalescent (negatively correlated) (see \cite{EJP2286}). To prove this theorem, we have to give the asymptotic behaviours of $\mathbb{E}[T_1^{(n)}T_2^{(n)}]$(Theorem \ref{2t}) and $\mathbb{E}[T_1^{(n)}]$ (Theorem \ref{1t}). We also get from Theorem \ref{1t} that  the third term in (\ref{decom2}) satisfies
\begin{equation}\label{thirdterm}(n\mathbb{E}[T_1^{(n)}]-n^{2-\alpha}\mathbb{E}[T])^2=O(n^{6-4\alpha}).\end{equation}

The second one gives the asymptotic behaviour of moments of one external branch length, hence we can estimate $nVar(T_1^{(n)})$.
\end{rem}
\begin{theo}\label{secondtheo}We assume that $\rho(t)$ satisfies (\ref{mainassum}) and that $ 2(\alpha-1)< \omega^{(1)}$, and $ \zeta>1-1/\alpha$.  
\begin{enumerate}
\item If $0\leq \beta<\frac{\alpha}{\alpha-1}, \beta(\alpha-1)<\omega^{(1)}$,  then $\displaystyle \lim_{n\to +\infty }\mathbb{E}[(n^{\alpha-1}T_1^{(n)})^{\beta}]=\mathbb{E}[T^{\beta}].$
\item If $\beta\geq \frac{\alpha}{\alpha-1}$, then $\displaystyle \lim_{n\to +\infty }\mathbb{E}[(n^{\alpha-1}T_1^{(n)})^{\beta}]=+\infty.$
\end{enumerate}
\end{theo}

\begin{rem} Under the assumptions of Theorem \ref{cor1}, we can apply Theorem \ref{secondtheo}. Since $\frac{\alpha}{\alpha-1}>2, \omega^{(1)}>2(\alpha-1)$, Theorem \ref{secondtheo} with $\beta=1, 2$ leads to  
$$\lim_{n\rightarrow +\infty}Var(n^{\alpha-1}T_1^{(n)})=Var(T)<+\infty.$$
Hence $nVar(T_1^{(n)})=O(n^{3-2\alpha})$. Recall that $(n\mathbb{E}[T_1^{(n)}]-n^{2-\alpha}\mathbb{E}[T])^2=O(n^{6-4\alpha})$ (see (\ref{thirdterm})) and $n(n-1)\mbox{Cov}(T_1^{(n)},T_2^{(n)})=O(n^{5-3\alpha})$(Theorem \ref{thm1}).  Notice that we have 
 $ 5-3\alpha>6-4\alpha>3-2\alpha.$  Hence we deduce that $n(n-1)\mbox{Cov}(T_1^{(n)},T_1^{(n)})$ is the dominant term in (\ref{decom2}) . Theorem \ref{cor1} is then proved.
\end{rem}

For the Beta$(2-\alpha,\alpha)$-coalescent, Theorem \ref{secondtheo} gives:
\begin{cor}\label{corsecondtheo}
For Beta$(2-\alpha, \alpha)$-coalescent, we have 
\begin{enumerate}
\item If $0\leq \beta<\frac{\alpha}{\alpha-1}$, then $\displaystyle \lim_{n\rightarrow +\infty}\mathbb{E}[(n^{\alpha-1}T_1^{(n)})^{\beta}]=\mathbb{E}[T^{\beta}]$.
\item If $\beta\geq \frac{\alpha}{\alpha-1}$, then $\displaystyle \lim_{n\rightarrow +\infty}\mathbb{E}[(n^{\alpha-1}T_1^{(n)})^{\beta}]=+\infty$.
\end{enumerate}
\end{cor}

\subsection{Organization of this paper}
{ We consider the ${\Lambda}$-coalescents which satisfy assumption (\ref{mainassum}). This includes $Beta(2-\alpha,\alpha)$-coalescents. In sections 2 and 3, we give estimates of $\mathbb{E}[T_1^{(n)}]$ and $\mathbb{E}[T_1^{(n)}T_2^{(n)}]$ respectively. Both $\mathbb{E}[T_1^{(n)}]$ and $\mathbb{E}[T_1^{(n)}T_2^{(n)}]$ satisfy  the same kind of recurrence  which allows to get their estimates and they lead to an estimate of $Cov(T_1^{(n)}, T_2^{(n)})$ in section 3. The main tool is Lemma \ref{toollem} given in appendix A. In section 4, we deal with Theorem \ref{secondtheo}.  
Section $5$ is the appendix where are given some proofs omitted before.} 

\section{First moment of $T_1^{(n)}$ by recursive method}
\subsection{The main result}
\begin{theo}\label{1t}
We assume that $\rho(t)$ satisfies $(\ref{mainassum})$.

\begin{enumerate}
\item If $0<\zeta<\alpha-1,\alpha-1+\zeta<\omega^{(1)},$ then 
$$\mathbb{E}[T_1^{(n)}]=\frac{\alpha-1}{C_0\Gamma(2-\alpha)}n^{1-\alpha}-\frac{C_1\Gamma(2-\alpha+\zeta)(\alpha-1)^2}{(C_0\Gamma(2-\alpha))^2(\alpha-1-\zeta)}n^{1-\alpha-\zeta}+o(n^{1-\alpha-\zeta}).$$
\item If $\zeta=\alpha-1,2(\alpha-1)<\omega^{(1)},$ then 


$$\mathbb{E}[T_1^{(n)}]=\frac{\alpha-1}{C_0\Gamma(2-\alpha)}n^{1-\alpha}-\frac{(\alpha-1)^2C_1}{(C_0\Gamma(2-\alpha))^2}n^{2(1-\alpha)}\ln n+o(n^{2(1-\alpha)}\ln n).$$
\item If $\zeta>\alpha-1,2(\alpha-1)<\omega^{(1)},$ then 
\begin{align*}
&\mathbb{E}[T_1^{(n)}]=\frac{\alpha-1}{C_0\Gamma(2-\alpha)}n^{1-\alpha}\\
&+\frac{(\alpha-1)^2}{C_0\Gamma(3-\alpha)}(\frac{\int_0^1((1-x)^{1-\alpha}-1-(\alpha-1)x)\nu^{(1)}(dx)}{C_0\Gamma(2-\alpha)}+\frac{(\alpha-1)C_2^{(1)}-C_2}{C_0\Gamma(2-\alpha)})n^{2(1-\alpha)}\\
&+o(n^{2(1-\alpha)}),\\
\end{align*}
where $\displaystyle C_2=\lim_{t\rightarrow 0}\int_t^1\rho(r)dr-\frac{C_0t^{1-\alpha}}{\alpha-1}, C_2^{(1)}=\lim_{t\rightarrow 0}\int_t^1\rho^{(1)}(r)dr-\frac{C_0t^{1-\alpha}}{\alpha-1},   \rho^{(1)}(t)=\int_t^1\nu^{(1)}(dx).$
\end{enumerate}
\end{theo}

\begin{rem}{For the  $Beta(2-\alpha,\alpha)$-coalescent, we have ${\omega}^{(1)}={\alpha}+1, \zeta=1$. Hence we are in case $(3).$ }\end{rem}
The idea is to use the recurrence satisfied by $T_1^{(n)}$(see \cite{EJP2286}): 
\begin{equation}\label{mainequavsuia}
\mathbb{E}[T_1^{(n)}]=\frac{1}{g_n}+\sum_{k=2}^{n-1}p_{n,k}\frac{k-1}{n}\mathbb{E}[T_1^{(k)}].
\end{equation}

All the three cases in Theorem \ref{1t} will be proved in the same way. To give an idea of what should be done, we take case $(1)$ as example. 
By defining $L=\frac{\alpha-1}{C_0\Gamma(2-\alpha)}, Q=\frac{C_1\Gamma(2-\alpha+\zeta)(\alpha-1)^2}{(C_0\Gamma(2-\alpha))^2(\alpha-1-\zeta)},$ we transform the recurrence $(\ref{mainequavsuia})$ to 
\begin{align}\label{case1}
\left(\mathbb{E}[n^{\alpha-1}T_1^{(n)}]-L\right)n^{\zeta}+Q&=\left(\frac{n^{\alpha-1}}{g_n}-(1-\sum_{k=1}^{n-1}p_{n,k}\frac{k-1}{n}(\frac{n}{k})^{\alpha-1})L\right)n^{\zeta}+Q(1-\sum_{k=1}^{n-1}p_{n,k}\frac{k-1}{n}(\frac{n}{k})^{\alpha-1+\zeta})\nonumber\\
&+\sum_{k=2}^{n-1}(\frac{n}{k})^{\alpha-1+\zeta}p_{n,k}\frac{k-1}{n}\left(k^{\zeta}(\mathbb{E}[k^{\alpha-1}T_1^{(k)}]-L)+Q\right).\\\nonumber
\end{align}

Hence we get a recurrence 
\begin{equation}\label{case1recu}a_n=b_n+\sum_{k=2}^{n-1}q_{n,k}a_k, \end{equation}
 with
$$a_n=\left(\mathbb{E}[n^{\alpha-1}T_1^{(n)}]-L\right)n^{\zeta}+Q,$$
$$b_n=\left(\frac{n^{\alpha-1}}{g_n}-(1-\sum_{k=2}^{n-1}p_{n,k}\frac{k-1}{n}(\frac{n}{k})^{\alpha-1})L\right)n^{\zeta}+Q(1-\sum_{k=2}^{n-1}p_{n,k}\frac{k-1}{n}(\frac{n}{k})^{\alpha-1+\zeta}),$$
$$q_{n,k}=(\frac{n}{k})^{\alpha-1+\zeta}p_{n,k}\frac{k-1}{n}.$$

Then we should prove that $\displaystyle \lim_{n\rightarrow +\infty}a_n=0.$
It is natural to think about estimating $b_n$ as $n$ tends to infinity. To this aim, we should get asymptotics of $g_n$, $\frac{1}{g_n}$,  and  $\sum_{k=1}^{n-1}p_{n,k}\frac{(k-1)_l}{(n)_l}(\frac{n}{k})^{r}$  with $r\geq 0$ and $l\in \mathbb{N}$, where $(n)_l$ is (the same for $(k-1)_l$):

\[(n)_l=\left\{ 
\begin{array}{ll}
n(n-1)(n-2)\cdots(n-l+1) & \mbox{if $n\geq l\geq 1$},\\
0 & \mbox{if $l>n\geq 1$.}
\end{array}\right.
\]

 \subsection{Asymptotics  of  ${1}/{g_n}$}
For any $ c,d \in \mathbb{R}$, we have
 \begin{equation}\label{ncd}
\frac{\Gamma(n+c)}{\Gamma(n+d)}=n^{c-d}(1+n^{-1}(c-d)\frac{c+d-1}{2}+O(n^{-2})).
\end{equation}
This is the straightforward consequence of Stirling's formula:
\begin{equation}\label{stirling}
\Gamma(z)=\sqrt{2\pi}z^{z-1/2}e^{-z}(1+\frac{1}{12z}+O(\frac{1}{z^2})), z>0.
\end{equation}

It is then easy to get:
For real numbers $a$ and $b>-1$,
\begin{equation}\label{f1}\int_0^1(1-t)^{n+a}t^bdx=\frac{\Gamma(n+a+1)\Gamma(b+1)}{\Gamma(n+a+b+2)}=\Gamma(b+1)n^{-1-b}(1+n^{-1}(-1-b)\frac{b+2a+2}{2}+O(n^{-2})).\end{equation}

Recall that $g_n=n(n-1)\int_0^1t(1-t)^{n-2}\rho(t)dt$. Hence using $(\ref{f1})$, we get the following lemma.
\begin{lem}
\begin{enumerate}
\item If $\rho(t)$ satisfies the condition:
\begin{equation}\label{2c}\rho(t)=C_0t^{-\alpha}+o(t^{-\alpha}),C_0>0. \end{equation} Then 
$$g_n=C_0\Gamma(2-\alpha)n^{\alpha}+o(n^{\alpha}),\quad \text{and} \quad \frac{1}{g_n}=\frac{1}{C_0\Gamma(2-\alpha)}n^{-\alpha}+o(n^{-\alpha}).$$
\item If $\rho(t)$ satisfies $(4)$, then
$$g_n=C_0\Gamma(2-\alpha)n^{\alpha}-\frac{\alpha(\alpha-1)}{2}C_0\Gamma(2-\alpha)n^{\alpha-1}+C_1\Gamma(2-\alpha+\zeta)n^{\alpha-\zeta}+o(n^{\alpha-\zeta})+O(1),$$
and
\begin{equation}\label{gninverse}
\frac{1}{g_n}=\left\{
\begin{array}{ll}
\left(1-\frac{C_1\Gamma(2-\alpha+\zeta)}{C_0\Gamma(2-\alpha)}n^{-\zeta}+o(n^{-\zeta})\right)\frac{n^{-\alpha}}{C_0\Gamma(2-\alpha)} & \mbox{if $0<\zeta<1$},\\
 \left(1+(\frac{\alpha(\alpha-1)}{2}-\frac{C_1\Gamma(2-\alpha+\zeta)}{C_0\Gamma(2-\alpha)})n^{-1}+o(n^{-1})\right)\frac{n^{-\alpha}}{C_0\Gamma(2-\alpha)} & \mbox{if $\zeta=1$},\\
\left(1+\frac{\alpha(\alpha-1)}{2}n^{-1}+o(n^{-1}) \right)\frac{n^{-\alpha}}{C_0\Gamma(2-\alpha)} & $\mbox{if $\zeta>1$} .$\\
\end{array} \right.
\end{equation}
\end{enumerate}
\end{lem}

\subsection{Calculus of $\sum_{k=1}^{n-1}p_{n,k}\frac{(k-1)_l}{(n)_l}(\frac{n}{k})^{r}, r\geq 0, l\in\mathbb{N}$}
For any $s\in \mathbb{R}$, we denote by 
$$g_{n}^{(s)}=\int_0^1(1-(1-x)^n-nx(1-x)^{n-1})\nu^{(s)}(dx){,}$$
the collision rates of the ${\Lambda}$-coalescent associated with the measure $\nu^{(s)}$. We recall that $\nu^{(s)}$ is defined in $(\ref{notation})$.  Notice that  $g_n^{(0)}=g_n$.

\begin{lem}\label{lemphk}
 Consider any $\Lambda$-coalescent process with measure $\nu$. Let $l\in\{1,2,\cdots,n-2\}$ fixed. Then for any real function $f$:
$$\sum_{k=1}^{n-1}p_{n,k}\frac{(k-1)_l}{(n)_l}f(k)=\mathbb{E}[\frac{(n-1-X_1^{(n)})_l}{(n)_l}]\mathbb{E}^{\nu^{(l)}}[f(n-X_1^{(n-l)})],$$
where $\mathbb{E}^{\nu^{(l)}}[*]$ means that the $\Lambda$-coalescent is associated with the measure $\nu^{(l)}$.
\end{lem}
\begin{proof}

Recall the definitions of $g_n$ and $p_{n,k}$(see (\ref{gn}), (\ref{pnk})). We have

\begin{align}\label{phk1}
\sum_{k=1}^{n-1}p_{n,k}\frac{(k-1)_l}{(n)_l}&=\sum_{k=l+1}^{n-1}\frac{\int_{0}^{1}{n-l\choose n-k+1}x^{n-k+1}(1-x)^{k-1}\nu(dx)}{g_{n}}\nonumber\\
&=\sum_{k=l+1}^{n-1}\frac{\int_{0}^{1}{n-l\choose n-k+1}x^{n-k+1}(1-x)^{k-1-l}\nu^{(l)}(dx)}{g_{n}}\nonumber\\
&=\sum_{k=1}^{n-1-l}\frac{\int_{0}^{1}{n-l\choose n-k-l+1}x^{n-k-l+1}(1-x)^{k-1}\nu^{(l)}(dx)}{g_{n}}=\frac{g^{(l)}_{n-l}}{g_n}.
\end{align}
Then,
\begin{align*}
\sum_{k=1}^{n-1}p_{n,k}\frac{(k-1)_l}{(n)_l}f(k)&=\left(\sum_{k=1}^{n-1}p_{n,k}\frac{(k-1)_l}{(n)_l}\right)\frac{\sum_{k=1}^{n-1}p_{n,k}\frac{(k-1)_l}{(n)_l}f(k)}{\sum_{k=1}^{n-1}p_{n,k}\frac{(k-1)_l}{(n)_l}}\\
&=\mathbb{E}[\frac{(n-1-X_1^{(n)})_l}{(n)_l}]\frac{\sum_{k=l+1}^{n-1}\int_0^1{n-l\choose n-k+1}x^{n-k+1}(1-x)^{k-1-l}f(k)\nu^{(l)}(dx)}{g^{(l)}_{n-l}}\\
&=\mathbb{E}[\frac{(n-1-X_1^{(n)})_l}{(n)_l}]\frac{\sum_{k=1}^{n-1-l}\int_0^1{n-l\choose n-k-l+1}x^{n-k-l+1}(1-x)^{k-1}f(k+l)\nu^{(l)}(dx)}{g^{(l)}_{n-l}}\\
&=\mathbb{E}[\frac{(n-1-X_1^{(n)})_l}{(n)_l}]\mathbb{E}^{\nu^{(l)}}[f(Y_1^{(n-l)}+l)]=\mathbb{E}[\frac{(n-1-X_1^{(n)})_l}{(n)_l}]\mathbb{E}^{\nu^{(l)}}[f(n-X_1^{(n-l)})].\\
\end{align*}
This achieves the proof of the lemma.\end{proof}

In consequence, 
\begin{equation}\label{2d}\sum_{k=1}^{n-1}p_{n,k}\frac{(k-1)_l}{(n)_l}(\frac{n}{k})^{r}=\mathbb{E}[\frac{(n-1-X_1^{(n)})_l}{(n)_l}]\mathbb{E}^{\nu^{(l)}}[(\frac{n}{n-X_1^{(n-l)}})^r].\end{equation} 

We have to study $\mathbb{E}[\frac{(n-1-X_1^{(n)})_l}{(n)_l}]$ and $\mathbb{E}^{\nu^{(l)}}[(\frac{n}{n-X_1^{(n-l)}})^r]$. The latter is very close to Proposition \ref{inverser} in appendix B. The following lemma studies the former. 

\begin{lem}\label{xl}
We assume that $\rho(t)$ satisfies $(\ref{mainassum})$.  Let $l\in\{1,2\cdots,n-2\}$ fixed. 
\begin{enumerate}
\item If $0<\zeta<\alpha-1$, then
\begin{align*}
&\mathbb{E}[\frac{(n-1-X_1^{(n)})_l}{(n)_l}]=1-\frac{l\alpha}{n(\alpha-1)}-\frac{lC_1\Gamma(2-\alpha+\zeta)\zeta}{C_0(\alpha-1)\Gamma(2-\alpha)(\alpha-1-\zeta)}n^{-\zeta-1}+o(n^{-\zeta-1}).
\end{align*}
\item If $\zeta=\alpha-1$, then
\begin{align*}
&\mathbb{E}[\frac{(n-1-X_1^{(n)})_l}{(n)_l}]=1-\frac{l\alpha}{n(\alpha-1)}-\frac{lC_1}{C_0\Gamma(2-\alpha)}n^{-\alpha}\ln n+o(n^{-\alpha}\ln n).\\
\end{align*}

\item If  $\zeta>\alpha-1$, then
\begin{align*}
&\mathbb{E}[\frac{(n-1-X_1^{(n)})_l}{(n)_l}]=1-\frac{l\alpha}{n(\alpha-1)}+\left(\sum_{j=2}^{l}{l\choose j}(-1)^j \frac{\int_0^1 j x^{j-1}\rho(x)dx}{C_0\Gamma(2-\alpha)}-\frac{C_2l}{C_0\Gamma(2-\alpha)}\right)n^{-\alpha}+o(n^{-\alpha}),
\end{align*}

where $C_2$ is the same as  in Theorem \ref{1t}.

\end{enumerate}
\end{lem}
\begin{proof} We have 
\begin{align*}
&\mathbb{E}[\frac{(n-1-X_1^{(n)})_l}{(n)_l}]=\mathbb{E}[1-\sum_{i=0}^{l-1}\frac{X_1^{(n)}+1}{n-i}+\sum_{j=2}^{l}\sum_{i_1, i_2, \cdots, i_j \text{all different}}(-1)^j\frac{(X_1^{(n)}+1)^j}{(n-i_1)(n-i_2)\cdots(n-i_j)}].\\
\end{align*}

{For $\mathbb{E}[\sum_{i=0}^{l-1}\frac{X_1^{(n)}+1}{n-i}]$, we use Lemma \ref{x1}} in appendix B. While using Lemme \ref{xk}, we get 
\begin{align}\label{xlalpha}
&\mathbb{E}[\sum_{j=2}^{l}\sum_{i_1, i_2, \cdots, i_j \text{all different}}(-1)^j\frac{(X_1^{(n)}+1)^j}{(n-i_1)(n-i_2)\cdots(n-i_j)}]=n^{-\alpha}\sum_{j=2}^{l}{l\choose j}(-1)^j \frac{\int_0^1 j x^{j-1}\rho(x)dx}{C_0\Gamma(2-\alpha)}+O(n^{-\min\{1+\alpha,j\}}).\nonumber\\
\end{align}
Then we conclude.

\end{proof}

Now we can finally give the estimate of $\sum_{k=1}^{n-1}p_{n,k}\frac{(k-1)_l}{(n)_l}(\frac{n}{k})^{r}$.

\begin{prop}\label{h1}
We assume that $\rho(t)$ satisfies (\ref{mainassum}). Let $l\in\{1,2,\cdots,n-2\}$ fixed and {$r\in[0, \omega^{(l)})$}. 

\begin{enumerate}

\item If $0<\zeta<\alpha-1$, then

\begin{align*}
&\sum_{k=1}^{n-1}p_{n,k}\frac{(k-1)_l}{(n)_l}(\frac{n}{k})^{r}=1+\frac{r-l\alpha}{n(\alpha-1)}+\frac{(r-l)C_1\Gamma(2-\alpha+\zeta)\zeta}{C_0(\alpha-1)\Gamma(2-\alpha)(\alpha-1-\zeta)}n^{-\zeta-1}+o(n^{-\zeta-1}).\\
\end{align*}

\item If $\zeta=\alpha-1, $ then
$$\sum_{k=1}^{n-1}p_{n,k}\frac{(k-1)_l}{(n)_l}(\frac{n}{k})^{r}=1+\frac{r-l\alpha}{n(\alpha-1)}+\frac{(r-l)C_1}{C_0\Gamma(2-\alpha)}n^{-\alpha}\ln n+o(n^{-\alpha}\ln n).$$
\item If $\zeta>\alpha-1,$ then
{\begin{align*}
\sum_{k=1}^{n-1}p_{n,k}\frac{(k-1)_l}{(n)_l}(\frac{n}{k})^{r}
=&1+\frac{(r-l\alpha)}{n(\alpha-1)}+(\frac{\int_0^1((1-x)^{-r}-1-rx)\nu^{(l)}(dx)}{C_0\Gamma(2-\alpha)}\\ 
&+\sum_{j=2}^{l}{l\choose j}(-1)^j \frac{\int_0^1 j x^{j-1}\rho(x)dx}{C_0\Gamma(2-\alpha)}+\frac{rC_2^{(l)}-lC_2}{C_0\Gamma(2-\alpha)})n^{-\alpha}+o(n^{-\alpha}),
\end{align*}
}
where $ C_2$ is the same as in Theorem \ref{1t} and $\displaystyle C^{(l)}_2=\lim_{t\rightarrow 0}\int_t^1\rho^{(l)}(r)dr-\frac{C_0t^{1-\alpha}}{\alpha-1}.$

\end{enumerate}
\end{prop}
\begin{proof}
The proof is a consequence of (\ref{2d}) and Remark~\ref{rem_ext}.
\end{proof}
\begin{rem}
The Remark \ref{rem2} following Proposition \ref{inverser} implies that if $r\geq \omega^{(r)}$, this proposition is not valid at least in case $(3)$. Since this proposition is critical for the proof of Theorem \ref{1t}(this will be shown in subsection $2.4$),  the assumptions about $\alpha$ and $\omega^{(1)}$ made in Theorem \ref{1t} are necessary.

\end{rem}

\subsection{Proof of Theorem \ref{1t}.}

\begin{enumerate}
\item
Recall the transformation $(\ref{case1})$ and the associated recurrence $(\ref{case1recu})$.  The aim is to prove that $\displaystyle \lim_{n\rightarrow +\infty}a_n=0$ for $a_n$ in $(\ref{case1recu})$.
 Under the assumptions  $0< \zeta<\alpha-1$, and $ \alpha-1+\zeta<\omega^{(1)}$,  using Proposition \ref{h1}, we get 
$$1-\sum_{k=1}^{n-1}p_{n,k}\frac{k-1}{n}(\frac{n}{k})^{\alpha-1}=\frac{1}{n(\alpha-1)}-\frac{(\alpha-2)C_1\Gamma(2-\alpha+\zeta)\zeta}{C_0(\alpha-1)\Gamma(2-\alpha)(\alpha-1-\zeta)}n^{-\zeta-1}+o(n^{-\zeta-1}),$$
$$1-\sum_{k=1}^{n-1}p_{n,k}\frac{k-1}{n}(\frac{n}{k})^{\alpha-1+\zeta}=\frac{1-\zeta}{n(\alpha-1)}-\frac{(\alpha-2+\zeta)C_1\Gamma(2-\alpha+\zeta)\zeta}{C_0(\alpha-1)\Gamma(2-\alpha)(\alpha-1-\zeta)}n^{-\zeta-1}+o(n^{-\zeta-1}),$$

and then $b_n=o(n^{-1})$ by $(\ref{gninverse})$. 
Let ${\varepsilon}>0$ such that $\alpha-1+\zeta+{\varepsilon}<\omega^{(1)}$. Then we have $1-\sum_{k=1}^{n-1}p_{n,k}\frac{k-1}{n}(\frac{n}{k})^{\alpha-1+\zeta+{\varepsilon}}=O(n^{-1})>0.$ Hence the recurrence $(\ref{case1recu})$ satisfies the assumptions of Lemma \ref{toollem} which leads to $\displaystyle \lim_{n\rightarrow +\infty} a_n=0.$ Then we can conclude.
\item  

Similarly, we transform recurrence (\ref{mainequavsuia}) to 
\begin{align*}
&\frac{n^{\alpha-1}}{\ln n}\left(\mathbb{E}[n^{\alpha-1}T_1^{(n)}]-L\right)+Q\\
&=\frac{n^{\alpha-1}}{\ln n}\left(\frac{n^{\alpha-1}}{g_n}-(1-\sum_{k=1}^{n-1}p_{n,k}\frac{k-1}{n}(\frac{n}{k})^{\alpha-1})L\right)+Q\left(1-\sum_{k=1}^{n-1}p_{n,k}\frac{k-1}{n}(\frac{n}{k})^{2(\alpha-1)}\frac{\ln k}{\ln n}\right)\\
&+\sum_{k=2}^{n-1}\frac{\ln k}{\ln n}p_{n,k}\frac{k-1}{n}(\frac{n}{k})^{2(\alpha-1)}\left(\frac{k^{\alpha-1}}{\ln k}(\mathbb{E}[k^{\alpha-1}T_1^{(k)}]-L)+Q\right),\\
\end{align*}
where $L=\frac{\alpha-1}{C_0\Gamma(2-\alpha)}, Q=\frac{(\alpha-1)^2C_1}{(C_0\Gamma(2-\alpha))^2}.$
Again, we denote by
$$a_n=\frac{n^{\alpha-1}}{\ln n}\left(\mathbb{E}[n^{\alpha-1}T_1^{(n)}]-L\right)+Q,$$
$$b_n=\frac{n^{\alpha-1}}{\ln n}\left(\frac{n^{\alpha-1}}{g_n}-(1-\sum_{k=1}^{n-1}p_{n,k}\frac{k-1}{n}(\frac{n}{k})^{\alpha-1})L\right)+Q\left(1-\sum_{k=1}^{n-1}p_{n,k}\frac{k-1}{n}(\frac{n}{k})^{2(\alpha-1)}\frac{\ln k}{\ln n}\right),$$
{$$q_{n,k}=(\frac{n}{k})^{2(\alpha-1)}\frac{\ln k}{\ln n}p_{n,k}\frac{k-1}{n},$$}
so that $a_n=b_n+\sum_{k=2}^{n-1}q_{n,k}a_k.$

Here the only thing that we do not know is the estimate of $\sum_{k=1}^{n-1}p_{n,k}\frac{k-1}{n}(\frac{n}{k})^{2(\alpha-1)}\frac{\ln k}{\ln n}$. Using the same idea as in Proposition \ref{h1}, we prove that : If $\rho(t)$ satisfies condition $(\ref{mainassum})$ with $\zeta=\alpha-1$, for $l$ fixed in $\{1,2,\cdots,n-2\}$ and $r\in[0,\omega^{(l)})$,
\begin{equation}\label{zeta=}\sum_{k=1}^{n-1}p_{n,k}\frac{(k-1)_l}{(n)_l}(\frac{n}{k})^{r}(\frac{\ln k}{\ln n})^{\beta}=1+\frac{r-l\alpha}{n(\alpha-1)}-\frac{\beta}{(\alpha-1)n\ln n}+\frac{(r-l)C_1}{C_0\Gamma(2-\alpha)}n^{-\alpha}\ln n+o(n^{-\alpha}\ln n).\end{equation}

 Every step is identical to the proof of the first case to get  $\displaystyle \lim_{n\rightarrow +\infty}a_n=0.$ We deduce that $\mathbb{E}[T_1^{(n)}]=Ln^{1-\alpha}-Qn^{2(1-\alpha)}\ln n+o(n^{2(1-\alpha)} \ln n).$

\item This case is similar to the first case, so we skip it for simplicity.

\end{enumerate}

\section{Estimate of $\mbox{Cov}(T_1^{(n)},T_2^{(n)})$}
Using Theorem 1.1 of \cite{EJP2286}, we have
\begin{equation}\label{2trecu}\mathbb{E}[T_1^{(n)}T_2^{(n)}]=\frac{2\mathbb{E}[T_1^{(n)}]}{g_n}+\sum_{k=1}^{n-1}p_{n,k}\frac{(k-1)_2}{(n)_2}\mathbb{E}[T_1^{(k)}T_2^{(k)}].
\end{equation}
The estimate of $\mathbb{E}[T_1^{(n)}T_2^{(n)}]$ can be obtained in the same way as in the proof of Theorem \ref{1t}.  At first, we give an estimate of $\frac{2\mathbb{E}[T_1^{(n)}]}{g_n}$ which is the consequence of (\ref{gninverse}) and Theorem \ref{1t}.  We assume that $\rho(t)$ satisfies (\ref{mainassum}).\begin{enumerate}
\item If $0<\zeta<\alpha-1, \alpha-1+\zeta<\omega^{(1)},$ then
$$\frac{2\mathbb{E}[T_1^{(n)}]}{g_n}=\frac{2n^{1-2\alpha}}{C_0\Gamma(2-\alpha)}(\frac{\alpha-1}{C_0\Gamma(2-\alpha)}-\frac{C_1\Gamma(2-\alpha+\zeta)(\alpha-1)(2\alpha-2-\zeta)}{(C_0\Gamma(2-\alpha))^2(\alpha-1-\zeta)}n^{-\zeta}+o(n^{-\zeta})).$$
\item If $\zeta=\alpha-1,2(\alpha-1)<\omega^{(1)},$ then
$$\frac{2\mathbb{E}[T_1^{(n)}]}{g_n}=\frac{2n^{1-2\alpha}}{C_0\Gamma(2-\alpha)}(\frac{\alpha-1}{C_0\Gamma(2-\alpha)}-\frac{C_1(\alpha-1)^2}{(C_0\Gamma(2-\alpha))^2}n^{1-\alpha}\ln n+o(n^{1-\alpha}\ln n)).$$
\item If  $\zeta>\alpha-1,2(\alpha-1)<\omega^{(1)},$ then
\begin{align*}
&\frac{2\mathbb{E}[T_1^{(n)}]}{g_n}=\frac{2n^{1-2\alpha}}{C_0\Gamma(2-\alpha)}\big(\frac{\alpha-1}{C_0\Gamma(2-\alpha)}\\
&+\frac{(\alpha-1)^2}{C_0\Gamma(3-\alpha)}(\frac{\int_0^1((1-x)^{1-\alpha}-1-(\alpha-1)x)\nu^{(1)}(dx)}{C_0\Gamma(2-\alpha)}+\frac{(\alpha-1)C_2^{(1)}-C_2}{C_0\Gamma(2-\alpha)})n^{1-\alpha}+o(n^{1-\alpha})\big).\\
\end{align*}

\end{enumerate}

We give the result and leave the proof to the reader.
\begin{theo}\label{2t}
We assume that $\rho(t)$ satisfies (\ref{mainassum}).\begin{enumerate}
\item If $0<\zeta<\alpha-1, (\alpha-1)+\zeta<\omega^{(1)},$ then


\begin{align*}
&\mathbb{E}[T_1^{(n)}T_2^{(n)}]=(\frac{\alpha-1}{C_0\Gamma(2-\alpha)})^2n^{2(1-\alpha)}-\frac{2C_1\Gamma(2-\alpha+\zeta)}{\alpha-1-\zeta}(\frac{\alpha-1}{C_0\Gamma(2-\alpha)})^3n^{2(1-\alpha)-\zeta}+o(n^{2(1-\alpha)-\zeta}).\\
\end{align*}

\item If $\zeta=\alpha-1,2(\alpha-1)<\omega^{(1)},$ then
\begin{align*}
&\mathbb{E}[T_1^{(n)}T_2^{(n)}]=(\frac{\alpha-1}{C_0\Gamma(2-\alpha)})^2n^{2(1-\alpha)}-\frac{2C_1(\alpha-1)^3}{(C_0\Gamma(2-\alpha))^3}n^{3(1-\alpha)}\ln n+o(n^{3(1-\alpha)}\ln n).\\
\end{align*}

\item If  $\zeta>\alpha-1,2(\alpha-1)<\omega^{(1)},$ then


\begin{align*}
&\mathbb{E}[T_1^{(n)}T_2^{(n)}]\\
&=(\frac{\alpha-1}{C_0\Gamma(2-\alpha)})^2n^{2(1-\alpha)}\\
&+\frac{\alpha-1}{3-\alpha}(\frac{\alpha-1}{C_0\Gamma(2-\alpha)})^2\left(B+\frac{2(\alpha-1)C_2^{(2)}+\int_0^12t\rho(t)dt-2C_2}{C_0\Gamma(2-\alpha)}+\frac{2}{2-\alpha}(A+\frac{(\alpha-1)C_2^{(1)}-C_2}{C_0\Gamma(2-\alpha)})\right)n^{3(1-\alpha)}\\
&\quad+o(n^{3(1-\alpha)}),\\
\end{align*}
where $ C_2$ and $\displaystyle C^{(l)}_2$ are the same as in Proposition \ref{h1}, and

$$\displaystyle A=\frac{\int_0^1((1-x)^{1-\alpha}-1-(\alpha-1)x)\nu^{(1)}(dx)}{C_0\Gamma(2-\alpha)},B=\frac{\int_0^1((1-x)^{2(1-\alpha)}-1-2(\alpha-1)x)\nu^{(2)}(dx)}{C_0\Gamma(2-\alpha)}.$$

\end{enumerate}
\end{theo}

Now combining Theorem \ref{1t} and \ref{2t}, we can get the estimate of $\mbox{Cov}(T_1^{(n)},T_2^{(n)}).$
\begin{cor}\label{thm4}
We assume that $\rho(t)$ satisfies (\ref{mainassum}).

\begin{enumerate}

\item If $0<\zeta<\alpha-1,\alpha-1+\zeta<\omega^{(1)},$ then 
$$\mbox{Cov}(T_1^{(n)},T_2^{(n)})=o(n^{2(1-\alpha)-\zeta}).$$

\item If $\zeta=\alpha-1,2(\alpha-1)<\omega^{(1)},$ then 
$$\mbox{Cov}(T_1^{(n)},T_2^{(n)})=o(n^{3(1-\alpha)}\ln n).$$

\item If $\zeta>\alpha-1,2(\alpha-1)<\omega^{(1)},$ then 
$$\mbox{Cov}(T_1^{(n)},T_2^{(n)})=\Delta(\alpha)n^{3(1-\alpha)}+o(n^{3(1-\alpha)}),$$
where $\Delta(\alpha)=\frac{\int_0^1((1-x)^{2-\alpha}-1)^2\nu(dx)}{3-\alpha}(\frac{\alpha-1}{C_0\Gamma(2-\alpha)})^3.$
\end{enumerate}

\end{cor}
\begin{proof}
We skip the easy cases $(1)$ and $(2)$. For case $(3)$, we get from Theorem \ref{1t} and \ref{2t} that 
$$\mbox{Cov}(T_1^{(n)},T_2^{(n)})=\frac{n^{3(1-\alpha)}}{3-\alpha}\frac{(\alpha-1)^3}{(C_0\Gamma(2-\alpha))^2}(B-2A+\frac{2(\alpha-1)(C_2^{(2)}-C_2^{(1)})+\int_0^12t\rho(t)dt}{C_0\Gamma(2-\alpha)})+o(n^{3(1-\alpha)}).$$
Hence 
\begin{equation}\label{delta}\Delta(\alpha)=\frac{(\alpha-1)^3}{(3-\alpha)(C_0\Gamma(2-\alpha))^2}(B-2A+\frac{2(\alpha-1)(C_2^{(2)}-C_2^{(1)})+\int_0^12t\rho(t)dt}{C_0\Gamma(2-\alpha)}).\end{equation}
Notice that 
$$B-2A=\frac{\int_0^1\left((1-x)^{2(2-\alpha)}-2(1-x)^{2-\alpha}+1-x^2+2(\alpha-1)x^2(1-x)\right)\nu(dx)}{C_0\Gamma(2-\alpha)}.$$
By definition, 
\begin{align*}
&C_2^{(2)}-C_2^{(1)}=\lim_{t\rightarrow +\infty}\int_t^1(\rho^{(2)}(x)-\rho^{(1)}(x))dx=\lim_{t\rightarrow 0}\int_t^1x(\nu^{(2)}(dx)-\nu^{(1)}(dx))=\int_0^1-x^2(1-x)\nu(dx),
\end{align*}
and  $\int_0^12t\rho(t)dt=\int_0^1x^2\nu(dx).$ Then we get the result.
\end{proof}

\subsection{Beta($2-\alpha,\alpha$)}

For the Beta$(2-\alpha,\alpha)$-coalescent, we have $\zeta=1, \omega^{(1)}=1+\alpha$, so $\zeta>\alpha-1, \omega^{(1)}>2(\alpha-1).$ Hence this process is in the case (3) of Corollary \ref{thm4}.  We can compute explicitly the constants $C_2,C_2^{(1)},C_2^{(2)}, A , B$ in Theorem \ref{2t}:
\begin{enumerate}
\item  $C_2=\frac{1}{1-\alpha}, C_2^{(1)}=\frac{\alpha}{1-\alpha},C_2^{(2)}=\frac{\alpha^2+\alpha}{2(1-\alpha)}$. 

\item  $A=\alpha(\alpha^2-\alpha-1)\Gamma(\alpha-1),$
 $B=\frac{1}{(\alpha-1)}(\frac{\Gamma(4-\alpha)}{\Gamma(4-2\alpha)}+(\alpha^2-\alpha-1)\Gamma(\alpha+2)).$
 \item $\Delta(\alpha)=\frac{\left((\alpha-1)\Gamma(\alpha+1)\right)^2\Gamma(4-\alpha)}{(3-\alpha)\Gamma(4-2\alpha)}.$
 \end{enumerate}

\begin{proof}
\begin{enumerate}
\item 
For Beta$(2-\alpha,\alpha)$-coalescent, we have for $C_0=\frac{1}{\alpha\Gamma(\alpha)\Gamma(2-\alpha)}$ and for $l =0,1,2$, 
$$\nu^{(l)}(dx)=\frac{1}{\Gamma(\alpha)\Gamma(2-\alpha)}x^{-1-\alpha}(1-x)^{\alpha-1+l}dx,\quad \text{and} \quad \rho^{(l)}(t)=\frac{1}{\alpha\Gamma(\alpha)\Gamma(2-\alpha)}t^{-\alpha}+O(t^{1-\alpha}).$$

Using integration by parts,  

\begin{align*}
{\int_t^1\rho^{(l)}(x)dx}&=-t\rho^{(l)}(t)+\frac{1}{\Gamma(\alpha)\Gamma(2-\alpha)}\int_t^1x^{-\alpha}(1-x)^{\alpha-1+l}dx\\
&=-\frac{1}{\alpha\Gamma(\alpha)\Gamma(2-\alpha)}t^{1-\alpha}+\frac{1}{(\alpha-1)\Gamma(\alpha)\Gamma(2-\alpha)}t^{1-\alpha}\\
&\quad+\frac{\alpha+l-1}{(1-\alpha)\Gamma(\alpha)\Gamma(2-\alpha)}\int_{t}^1x^{1-\alpha}(1-x)^{\alpha+l-2}dx+O(t^{2-\alpha})\\
&=\frac{C_0}{\alpha-1}t^{1-\alpha}+\frac{\Gamma(\alpha+l)}{\Gamma(l+1)\Gamma(\alpha)(1-\alpha)}+O(t^{2-\alpha}).\\
\end{align*}

The choices $l=0,1,2,$ give  $(1).$

\item 
Recall that 
$$\displaystyle A=\frac{\int_0^1((1-x)^{1-\alpha}-1-(\alpha-1)x)\nu^{(1)}(dx)}{C_0\Gamma(2-\alpha)}=\frac{\alpha}{\Gamma(2-\alpha)}\int_0^1((1-x)^{1-\alpha}-1-(\alpha-1)x)x^{-1-\alpha}(1-x)^{\alpha}dx.$$

Using integration by parts two times, 
\begin{align*}
A&=\frac{\alpha}{\Gamma(2-\alpha)}\frac{1}{\alpha(\alpha-1)}\int_0^1x^{1-\alpha}\left(-\alpha(\alpha-1)(1-x)^{\alpha-2}+2\alpha (\alpha-1)(1-x)^{\alpha-1}-\alpha(\alpha-1)^2x(1-x)^{\alpha-2}\right)dx\\
&=\frac{1}{\Gamma(2-\alpha)(\alpha-1)}\left(-\Gamma(\alpha+1)\Gamma(2-\alpha)+2(\alpha-1)\Gamma(\alpha+1)\Gamma(2-\alpha)-(\alpha-1)\Gamma(3-\alpha)\Gamma(\alpha+1) \right)\\
&=\alpha(\alpha^2-\alpha-1)\Gamma(\alpha-1).
\end{align*}

$B$ can be computed directly using integration by parts two times. 
\item Notice that $\Delta(\alpha)$ has the expression $(\ref{delta})$. Hence it can be calculated since we know $A,B,C_2^{(1)}$, and  $C_2^{(2)}$. We omit the very detailed calculus.
\end{enumerate}
\end{proof}

\section{Proof of Theorem~\ref{secondtheo}}
Notice that $n^{\alpha-1}T_1^{(n)}\stackrel{(d)}{\rightarrow} T$ and if $\beta\geq \frac{\alpha}{\alpha-1}$, one gets $\mathbb{E}[T^{\beta}]=+\infty$, hence $\mathbb{E}[(n^{\alpha-1}T_{1}^{(n)})^{\beta}]$ converges to $+\infty$ ({see Lemma 4.11 of \cite{MR1876169}}).
{
If $0\leq \beta_1< \beta_2<\frac{\alpha}{\alpha-1}$ and $(\mathbb{E}[(n^{\alpha-1}T_{1}^{(n)})^{\beta_2}],n\geq 2)$ is bounded. Then  $((n^{\alpha-1}T_1^{(n)})^{\beta_1},{n\geq 2})$ is uniformly integrable (see Lemma 4.11 of \cite{MR1876169} and  Problem 14 in section 8.3  \cite{breiman1992probability}). Then we need only to prove that for $\beta\in[2,\frac{\alpha}{\alpha-1})$, $(\mathbb{E}[(n^{\alpha-1}T_{1}^{(n)})^{\beta}],n\geq 2)$ is bounded.

We will prove by induction on $n$ that there exists a constant $C>0$ such that for all $n\geq 2$, $(\mathbb{E}[n^{\alpha-1}T_1^{(n)}])^{\beta}\leq C$. We first assume that, for all  $2\leq k\leq n-1$, $(\mathbb{E}[k^{\alpha-1}T_1^{(k)}])^{\beta}\leq C$ and then will prove that (if $C$ is large enough) $(\mathbb{E}[n^{\alpha-1}T_1^{(n)}])^{\beta}\leq C$. 
}

Writing the decomposition of $T_1^{(n)}$ at the first coalescence, we have
\[
T_{1}^{(n)}=\frac{e_{0}}{g_{n}}+\sum_{k=2}^{n-1}\mathbf{1}_{\{H_{n, k}\}}\bar{T}_{1}^{(k)},
\]
where:
\begin{itemize}
\item $\displaystyle H_{n,k}$ is the event: \{From $n$ individuals, we have  $k$  individuals after the first coalescence, and individual $ 1$ is not involved in this collision\}, $2\leq k\leq n-1$;
\item $e_{0}$ is a unit exponential random variable, $\bar{T}_{1}^{(k)}\stackrel{(d)}{=}T_{1}^{(k)}$, and all these random variables $e_0$, $\bar{T}_{1}^{(k)}$, $\mathbf{1}_{\{H_{n,k}\}}$ are independent. One notices that $\mathbb{P}(H_{n,k})=p_{n,k}\frac{k-1}{n}$, compared to $(\ref{mainequavsuia})$.
\end{itemize}

Thanks to Lemma \ref{inequality} in Appendix C, we have the following inequality.

\begin{align}\label{decom1}
 \mathbb{E}[(T_{1}^{(n)})^{\beta}]&=\mathbb{E}[((\frac{e_{0}}{g_{n}}+\sum_{k=2}^{n-1}\mathbf{1}_{\{H_{n,k}\}}\bar{T}_{1}^{(k)}))^{\beta}]\leq I_{n,1}+I_{n,2}+I_{n,3}+I_{n,4}
\end{align}

where
\begin{align*}
I_{n,1}&=\mathbb{E}[(\frac{e_{0}}{g_{n}})^{\beta}],\quad I_{n,2}=\mathbb{E}[(\sum_{k=2}^{n-1}\mathbf{1}_{\{H_{n,k}\}}\bar{T}_{1}^{(k)})^{\beta}],\\
I_{n,3}&=\mathbb{E}[\beta2^{\beta-1}\frac{e_{0}}{g_{n}}(\sum_{k=2}^{n-1}\mathbf{1}_{\{H_{n,k}\}}\bar{T}_{1}^{(k)})^{\beta-1}] \mbox{ and } I_{n,4}=\mathbb{E}[\beta2^{\beta-1}(\frac{e_{0}}{g_{n}})^{\beta-1}\sum_{k=2}^{n-1}\mathbf{1}_{\{H_{n,k}\}}\bar{T}_{1}^{(k)}].
\end{align*}

We first bound $I_{n,1}$. Recall that $\displaystyle g_n\sim C_0\Gamma(2-\alpha)n^{\alpha}$, hence there exists a constant $K_1>0$ such that for ${\beta}\geq 2>1$ and $n\geq 2$, 

\begin{eqnarray}
 \label{I1}
 n^{(\alpha-1)\beta}I_{n,1}\leq \frac{K_1}{n}.
\end{eqnarray}

We now consider $I_{n,2}$.
Notice that $(\alpha-1)\beta<\omega^{(1)}$. Hence, using Proposition \ref{h1}, we have

\begin{align}\label{I2}
n^{(\alpha-1)\beta}I_{n,2}
&=
n^{-(\alpha-1)\beta}\sum_{k=2}^{n-1}p_{n,k}\frac{k-1}{n}(\frac{n}{k})^{(\alpha-1)\beta}\mathbb{E}[(k^{\alpha-1}T_1^{(k)})^{\beta}]\\
&\leq 
C\sum_{k=2}^{n-1}p_{n,k}\frac{k-1}{n}(\frac{n}{k})^{(\alpha-1)\beta}\\
&=C(1-\frac{\alpha-(\alpha-1)\beta}{n(\alpha-1)}+o(n^{-1}))
&\leq C(1-\frac{\alpha-(\alpha-1)\beta}{2n(\alpha-1)}),
\end{align}
for $n\geq N,$ where $N$ is a fixed positive integer.

We now proceed to $I_{n,3}$. Notice that for $2\leq k\leq n-1,$ 
\[
\mathbb{E}[(k^{\alpha-1}T_{1}^{(k)})^{\beta-1}]\leq(\mathbb{E}[(k^{\alpha-1}T_{1}^{(k)})^{\beta}])^{\frac{\beta-1}{\beta}}\leq C^{\frac{\beta-1}{\beta}}.
\]
Hence we have
{\begin{align}\label{I3}
n^{(\alpha-1)\beta}I_{n,3}
&=n^{(\alpha-1)\beta}\mathbb{E}[\beta2^{\beta-1}\frac{e_{0}}{g_{n}}\sum_{k=2}^{n-1}\mathbf{1}_{\{H_{n,k}\}}(\bar{T}_{1}^{(k)})^{\beta-1}]\nonumber\\
&
\leq
C^{\frac{\beta-1}{\beta}}\beta2^{\beta-1}n^{\alpha-1}g_{n}^{-1}\sum_{k=2}^{n-1}p_{n,k}\frac{k-1}{n}\left(\frac{n}{k}\right)^{(\alpha-1)(\beta-1)}\nonumber\\
&
=
C^{\frac{\beta-1}{\beta}}n^{\alpha-1}\beta2^{\beta-1}g_{n}^{-1}(1-\frac{\alpha-(\alpha-1)(\beta-1)}{n(\alpha-1)}+o(n^{-1}))\nonumber\\
&\leq \frac{C^{\frac{\beta-1}{\beta}}K_2}{n},
\end{align}}
where $K_2$ is a positive constant. In the second equality, we have used Proposition \ref{h1}.

While for any $n\geq 2,$
\begin{align}\label{I4}
n^{(\alpha-1)\beta}I_{n,4}
&=n^{(\alpha-1)\beta}\mathbb{E}[\beta2^{\beta-1}(\frac{e_{0}}{g_{n}})^{\beta-1}\sum_{k=2}^{n-1}\mathbf{1}_{\{H_{n,k}\}}\bar{T}_{1}^{(k)}]\nonumber\\
&
\leq \beta2^{\beta-1} {\mathbb E}[e_{0}^{\beta-1}] (g_n)^{1-{\beta}}n^{(\alpha-1)(\beta-1)}{\mathbb E}[n^{\alpha-1}{T}_{1}^{(n)}]\nonumber\\
&\leq \frac{K_3}{n^{\beta-1}}\leq \frac{K_3}{n},
\end{align}
where $K_3$ is a positive constant. We have used Lemma \ref{1t} to bound $\mathbb{E}[n^{\alpha-1}T_1^{(n)}]$.

Using (\ref{decom1}),(\ref{I1}),(\ref{I2}),(\ref{I3}),(\ref{I4}), we have proved that for any $n, n\geq N$, if there exists $C>0$ such that for {all  $ 2\leq k \leq n-1$,  $\mathbb{E}[\left(k^{\alpha-1}T_{1}^{(k)}\right)^{\beta}]\leq C$}, then
\begin{equation}\label{cc}
\mathbb{E}[(n^{\alpha-1}T_{1}^{(n)})^{\beta}]\leq \frac{C+ \left(K_1-C\frac{\alpha-(\alpha-1)\beta}{2(\alpha-1)}+C^{\frac{\beta-1}{\beta}}K_2+K_3\right)}{n}.\end{equation}

Let $C$ large enough such that \begin{equation}\label{ccc} K_1-C\frac{\alpha-(\alpha-1)\beta}{2(\alpha-1)}+C^{\frac{\beta-1}{\beta}}K_2+K_3<0,\end{equation}
Then $\mathbb{E}[(n^{\alpha-1}T_{1}^{(n)})^{\beta}]\leq C$, which allows to conclude.\vspace{5mm}

\section{Appendix}
A) The main recurrence tool
\begin{lem}\label{toollem}
We consider the recurrence $a_n=b_n+\sum_{k=1}^{n-1}q_{n,k}a_k$. We assume that  $b_n=o(n^{-1})$ and that there exist  ${\varepsilon}>0$ and $C>0$ such that $1-\sum_{k=1}^{n-1}q_{n,k}(\frac{n}{k})^{{\varepsilon}}\geq Cn^{-1}$ for $n$ large enough. Then
$\displaystyle \lim_{n\rightarrow +\infty}a_n=0$.\end{lem}
\begin{proof}
{Let $(\bar{c}_{n})_{n\geq 1}$ be an increasing sequence such that}
$$\lim_{n\rightarrow +\infty}\bar{c}_n=+\infty; \lim_{n\rightarrow +\infty}nb_n\bar{c}_n=0.$$

Define another sequence $(c_n)_{n\geq 1}$ by: $c_{1}=\bar{c}_{1}$. For $n\geq 1$,
$$ c_{n+1}=\min\{c_n(\frac{n+1}{n})^{{\varepsilon}}, \bar{c}_{n+1}\},$$ 

Then we have $\displaystyle \lim_{n\rightarrow +\infty}c_n=+\infty, c_nb_n=o(n^{-1})$ and for any $1\leq k\leq n-1$, $\frac{c_n}{c_k}\leq (\frac{n}{k})^{\varepsilon}$. In consequence, $1-\sum_{k=1}^{n-1}q_{n,k}\frac{c_n}{c_k}\geq Cn^{-1}$ for $n$ large enough. We suppose that there exist $n_1>0$ such that for $n>n_1$, we have $1-\sum_{k=1}^{n-1}q_{n,k}\frac{c_n}{c_k}>\frac{C}{n}$ and $c_nb_n<\frac{C}{2n}$. We can find a number $C'$ such that $C'>\max\{1, c_ka_k ; 1\leq k\leq n_1\}.$ 
We transform the original recurrence to 
$$c_na_n=c_nb_n+\sum_{k=1}^{n-1}\left(q_{n,k}\frac{c_n}{c_k}\right)c_ka_k.$$
Then $c_{n_1+1}a_{n_1+1}\leq \frac{C}{2(n_1+1)}+(1-\frac{C}{n_1+1})C'\leq C'.$ By induction, we prove that the sequence $(c_na_n)_{n\geq 1}$ is bounded by $C'$. Since $c_n$ tends to the infinity, we get $\displaystyle \lim_{n\rightarrow +\infty}a_n=0$.
\end{proof}
\begin{rem}
This kind of  recurrence relationships is very frequent in probability. We refer to \cite{marynych2011stochastic} for a rather detailed survey.
\end{rem}
B) Asymptotic behaviours of $X_1^{(n)}$
\begin{lem}\label{x1}
We assume that $\rho(t)$ satisfies $(\ref{mainassum})$.
\begin{enumerate}
\item If $0<\zeta<\alpha-1,$ then $\mathbb{E}[X_1^{(n)}]=\frac{1}{\alpha-1}(1+\frac{C_1\Gamma(2-\alpha+\zeta)\zeta}{C_0\Gamma(2-\alpha)(\alpha-1-\zeta)}n^{-\zeta})+o(n^{-\zeta}).$
\item If $\zeta=\alpha-1,$ then $\mathbb{E}[X_1^{(n)}]=\frac{1}{\alpha-1}(1+\frac{C_1(\alpha-1)}{C_0\Gamma(2-\alpha)}n^{1-\alpha}\ln n)+o(n^{1-\alpha}\ln n).$

\item If  $\zeta>\alpha-1$, then $\mathbb{E}[X_1^{(n)}]=\frac{1}{\alpha-1}(1+\frac{C_2(\alpha-1)}{C_0\Gamma(2-\alpha)}n^{1-\alpha})+o(n^{1-\alpha}),$ where $\displaystyle C_2=\lim_{t\rightarrow 0}\int_t^1\rho(r)dr-\frac{C_0t^{1-\alpha}}{\alpha-1}.$
\end{enumerate}
\end{lem}
\begin{proof}
We have:
$$\mathbb{E}[X_1^{(n)}]=\frac{\int_0^1(1-t)^{n-2}(\int_t^1\rho(r)dr)dt}{\int_0^1(1-t)^{n-2}t\rho(t)dt}$$(see \cite{DDS2008}).
Hence we should give the values of $\int_0^1(1-t)^{n-2}(\int_t^1\rho(r)dr)dt$ and $\int_0^1(1-t)^{n-2}t\rho(t)dt$. For that, we use the estimation $(\ref{ncd}).$
\begin{enumerate}
\item If $0<\zeta<\alpha-1, $
 $\int_t^1\rho(r)dr=\int_t^1C_0r^{-\alpha}+C_1r^{-\alpha+\zeta}+o(r^{-\alpha+\zeta})dr=\frac{C_0}{\alpha-1}t^{1-\alpha}+\frac{C_1}{\alpha-1-\zeta}t^{-\alpha+\zeta+1}+o(t^{-\alpha+\zeta+1}).$
 Then using that $\alpha-\zeta-1>0$, (\ref{ncd}) and (\ref{f1}), we have
 \begin{align*}
 \int_0^1(1-t)^{n-2}(\int_t^1\rho(r)dr)dt&=\frac{C_0}{\alpha-1}\frac{\Gamma(n-1)\Gamma(2-\alpha)}{\Gamma(n+1-\alpha)}+(\frac{C_1}{\alpha-1-\zeta}+o(1))\frac{\Gamma(n-1)\Gamma(2-\alpha+\zeta)}{\Gamma(n+1-\alpha+\zeta)}\\
 &=\frac{C_0\Gamma(2-\alpha)}{\alpha-1}n^{\alpha-2}(1+O(n^{-1}))+(\frac{C_1}{\alpha-1-\zeta}+o(1))\Gamma(2-\alpha+\zeta)n^{\alpha-2-\zeta}(1+O(n^{-1}))\\
 &=\frac{C_0\Gamma(2-\alpha)}{\alpha-1}n^{\alpha-2}+\frac{C_1\Gamma(2-\alpha+\zeta)}{\alpha-1-\zeta}n^{\alpha-2-\zeta}+o(n^{\alpha-2-\zeta}),\\
 \end{align*}
 
 and
 \begin{align*}
 \int_0^1(1-t)^{n-2}t\rho(t)dt&=C_0\frac{\Gamma(n-1)\Gamma(2-\alpha)}{\Gamma(n+1-\alpha)}+(C_1+o(1))\frac{\Gamma(n-1)\Gamma(2-\alpha+\zeta)}{\Gamma(n+1-\alpha+\zeta)}\\
 &=C_0\Gamma(2-\alpha)n^{\alpha-2}(1+O(n^{-1}))+(C_1+o(1))\Gamma(2-\alpha+\zeta)n^{\alpha-2-\zeta}(1+O(n^{-1}))\\
 &=C_0\Gamma(2-\alpha)n^{\alpha-2}+C_1\Gamma(2-\alpha+\zeta)n^{\alpha-2-\zeta}+o(n^{\alpha-2+\zeta}).\\
 \end{align*}
 
 Hence $\mathbb{E}[X_1^{(n)}]=\frac{1}{\alpha-1}(1+\frac{C_1\Gamma(2-\alpha+\zeta)\zeta}{C_0\Gamma(2-\alpha)(\alpha-1-\zeta)}n^{-\zeta}+o(n^{-\zeta})).$

\item If $\zeta=\alpha-1$,  $\int_t^1\rho(r)dr=\frac{C_0}{\alpha-1}t^{1-\alpha}-C_1\ln t+o(\ln t)$. Then 
\begin{align*}
 &\int_0^1(1-t)^{n-2}(\int_t^1\rho(r)dr)dt=\frac{C_0}{\alpha-1}\frac{\Gamma(n-1)\Gamma(2-\alpha)}{\Gamma(n+1-\alpha)}+(-C_1+o(1))\int_0^1(1-t)^{n-2}\ln tdt.\\
  \end{align*}
Let $a_n=\int_0^1(1-t)^{n}\ln t dt$. By integration by parts, $(n+1)a_n=na_{n-1}-\frac{1}{n+1}$. So $a_n=\frac{-\sum_{i=2}^{n+1}1/i}{n+1}=\frac{-\ln n}{n}+O(n^{-1})$.

In consequence, 
\begin{align*}
\int_0^1(1-t)^{n-2}(\int_t^1\rho(r)dr)dt&=\frac{C_0\Gamma(2-\alpha)}{\alpha-1}n^{\alpha-2}(1+O(n^{-1}))+(-C_1+o(1))(\frac{-\ln (n-2)}{n-2}+O(n^{-1}))\\
 &=\frac{C_0\Gamma(2-\alpha)}{\alpha-1}n^{\alpha-2}+C_1\frac{\ln n}{n}+o(\frac{\ln n}{n}).\\
 \end{align*}
Moreover,
 \begin{align*}
 &\int_0^1(1-t)^{n-2}t\rho(t)dt=C_0\Gamma(2-\alpha)n^{\alpha-2}+C_1n^{-1}+o(n^{-1}).\\
 \end{align*}

Hence, $\mathbb{E}[X_1^{(n)}]=\frac{1}{\alpha-1}(1+\frac{C_1(\alpha-1)}{C_0\Gamma(2-\alpha)}n^{1-\alpha}\ln n)+o(n^{1-\alpha}\ln n).$

\item If $\zeta>\alpha-1,$ $\int_t^1\rho(r)dr=\frac{C_0}{\alpha-1}t^{1-\alpha}+C_2+o(1)$, where we recall that $\displaystyle C_2=\lim_{t\rightarrow 0}\int_t^1\rho(r)dr-\frac{C_0t^{1-\alpha}}{\alpha-1}$. Notice that for the term $o(1)$ here,  $\int_0^1(1-t)^{n-2}o(1)dt=o(n^{-1})$. Hence,
\begin{align*}
 \int_0^1(1-t)^{n-2}(\int_t^1\rho(r)dr)dt&=\frac{C_0\Gamma(2-\alpha)}{\alpha-1}n^{\alpha-2}+\frac{C_2}{n}+o(n^{-1}),\\
 \end{align*}
 
and,
\begin{align*}
 &\int_0^1(1-t)^{n-2}t\rho(t)dt=C_0\Gamma(2-\alpha)n^{\alpha-2}+O(n^{\alpha-2+\max\{-1,-\zeta\}}).
  \end{align*}

Hence we get  {$\mathbb{E}[X_1^{(n)}]=\frac{1}{\alpha-1}(1+\frac{C_2(\alpha-1)}{C_0\Gamma(2-\alpha)}n^{1-\alpha}+o(n^{1-\alpha})).$}

\end{enumerate}
\end{proof}

\begin{lem}\label{xk}
We assume that $\rho(t)$ satisfies $(\ref{2c})$. If $k\geq 2$, 
$$\mathbb{E}[(\frac{X_1^{(n)}}{n})^k]=\frac{\int_{0}^1kt^{k-1}\rho(t)dt}{C_0\Gamma(2-\alpha)}n^{-\alpha}+O(n^{-\min\{1+\alpha,k\}}).$$
\end{lem}
\begin{proof}
Let $B_{n,x}$ denote a binomial random variable with parameter $(n,x), n\geq 2, 0\leq x\leq 1$. Recall that for $2\leq i\leq n$,  $\mathbb{P}(X_1^{(n)}=i-1)=\int_0^1{n\choose i}x^{i}(1-x)^{n-i}\nu(dx)/g_n=\int_0^1\mathbb{P}(B_{n,x}=i)\nu(dx)/g_n$.

\begin{flalign*}
\mathbb{E}[(\frac{X_1^{(n)}}{n})^k]&=\int_0^1\mathbb{E}[(\frac{B_{n,x}-1}{n})^k\mathbf{1}_{B_{n,x}\geq 1}])\nu(dx)/g_n\\
&=\int_0^1n^{-k}\mathbb{E}[(B_{n,x}^k-B_{n,x})\\
&\qquad  \qquad +\sum_{i=1}^{k-1}{k\choose i}(-1)^i(B_{n,x}^{k-i}-B_{n,x})+(-1)^{k}(1-B_{n,x})\mathbf{1}_{B_{n,x}\geq  1})]\nu(dx)/g_n.
\end{flalign*}

Using Lemma \ref{bin} in Appendix C, we get $\mathbb{E}[(B_{n,x}^k-B_{n,x})]=(nx)^k+O(n^{k-1})x^2.$ Then

\begin{align*}
\mathbb{E}[(\frac{X_1^{(n)}}{n})^k]&=\int_0^1n^{-k}\left( (nx)^k+ O(n^{k-1})x^2\right)\nu(dx)/g_n+n^{-k}\int_0^1(-1)^{k}(1-nx-(1-x)^n)\nu(dx)/g_n\\
&=\frac{\int_{0}^1x^{k}\nu(dx)n^{-\alpha}}{C_0\Gamma(2-\alpha)}+O(n^{-\min\{1+\alpha,k\}})=\frac{\int_{0}^1kt^{k-1}\rho(t)dtn^{-\alpha}}{C_0\Gamma(2-\alpha)}+O(n^{-\min\{1+\alpha,k\}}).
\end{align*}

In the second equality, we have used $g_n\sim C_0\Gamma(2-\alpha)n^{\alpha}$ and also the fact that $\int_0^1(1-nx-(1-x)^n)\nu(dx)\leq g_n=\int_0^1(1-nx(1-x)^{n-1}-(1-x)^n)\nu(dx).$
This achieves the proof.

\end{proof}

\begin{prop}\label{inverser}
We assume that $\rho(t)$ satisfies $(\ref{mainassum})$ and $r\in[0,\omega)$. 
\begin{enumerate}
\item If $0<\zeta<\alpha-1$, then $\mathbb{E}[\left(\frac{n}{n-X_1^{(n)}}\right)^r]=1+\frac{r}{n(\alpha-1)}+\frac{rC_1\Gamma(2-\alpha+\zeta)\zeta}{C_0(\alpha-1)\Gamma(2-\alpha)(\alpha-1-\zeta)}n^{-\zeta-1}+o(n^{-\zeta-1}).
$
\item If $\zeta=\alpha-1$, then $\mathbb{E}[\left(\frac{n}{n-X_1^{(n)}}\right)^r]=1+\frac{r}{n(\alpha-1)}+\frac{rC_1}{C_0\Gamma(2-\alpha)}n^{-\alpha}\ln n+o(n^{-\alpha}\ln n).$
\item If $\zeta>\alpha-1$, then $\mathbb{E}[\left(\frac{n}{n-X_1^{(n)}}\right)^r]=1+\frac{r}{n(\alpha-1)}+(\frac{\int_0^1((1-x)^{-r}-1-rx) \nu(dx)}{C_0\Gamma(2-\alpha)}+\frac{rC_2}{C_0\Gamma(2-\alpha)})n^{-\alpha}+o(n^{-\alpha}),$
where $C_2$ is the same as in Theorem \ref{1t}.
\end{enumerate}
\end{prop}
\begin{proof}

By Taylor expansion formula, for $m\geq 2$, we have,

\begin{align*}
\left(\frac{n}{n-X_1^{(n)}}\right)^{r}&=\left(\frac{1}{1-\frac{X_1^{(n)}}{n}}\right)^r\\
&=1+r\frac{X_1^{(n)}}{n}+\sum_{k=2}^{m}\frac{\prod_{i=0}^{k-1}(r+i)}{k!}(\frac{X_1^{(n)}}{n})^k+\frac{\prod_{i=0}^{m}(r+i)}{m!}\int_0^{\frac{X_1^{(n)}}{n}}(1-t)^{-r-m-1}(\frac{X_1^{(n)}}{n}-t)^mdt.
\end{align*}
 Then we get:
\begin{equation}\label{**}
\mathbb{E}[\left(\frac{n}{n-X_1^{(n)}}\right)^{r}]=1+r\mathbb{E}[\frac{X_1^{(n)}}{n}]+\sum_{k=2}^{m}\frac{\prod_{i=0}^{k-1}(r+i)}{k!}\mathbb{E}[(\frac{X_1^{(n)}}{n})^k]+\frac{\prod_{i=0}^{m}(r+i)}{m!}\mathbb{E}[\int_0^{\frac{X_1^{(n)}}{n}}(1-t)^{-r-m-1}(\frac{X_1^{(n)}}{n}-t)^mdt].
\end{equation}

Thanks to Lemma \ref{x1} and  \ref{xk}, we have the asymptotic behaviours of $\mathbb{E}[\frac{X_1^{(n)}}{n}]$ and $\mathbb{E}[\sum_{k=2}^{m}\frac{\prod_{i=0}^{k-1}(r+i)}{k!}(\frac{X_1^{(n)}}{n})^k]$. In particular, for $m\geq 2$, 
$$\displaystyle \lim_{n\rightarrow +\infty}n^{\alpha}\mathbb{E}[\sum_{k=2}^{m}\frac{\prod_{i=0}^{k-1}(r+i)}{k!}(\frac{X_1^{(n)}}{n})^k]=\sum_{k=2}^{m}\frac{\prod_{i=0}^{k-1}(r+i)}{k!}\frac{\int_{0}^1kx^{k-1}\rho(t)dt}{C_0\Gamma(2-\alpha)}=\sum_{k=2}^{m}\frac{\prod_{i=0}^{k-1}(r+i)}{k!}\frac{\int_{0}^1x^{k}\nu(x)}{C_0\Gamma(2-\alpha)}.$$
In consequence,

$$\displaystyle \lim_{m\rightarrow +\infty}\lim_{n\rightarrow +\infty}n^{\alpha}\mathbb{E}[\sum_{k=2}^{m}\frac{\prod_{i=0}^{k-1}(r+i)}{k!}(\frac{X_1^{(n)}}{n})^k]= \lim_{m\rightarrow +\infty}\sum_{k=2}^{m}\frac{\prod_{i=0}^{k-1}(r+i)}{k!}\frac{\int_{0}^1x^{k}\nu(x)}{C_0\Gamma(2-\alpha)}=\frac{\int_0^1((1-x)^{-r}-1-rx) \nu(dx)}{C_0\Gamma(2-\alpha)}.$$
we need only to estimate $\frac{\prod_{i=0}^{m}(r+i)}{m!}\mathbb{E}[\int_0^{\frac{X_1^{(n)}}{n}}(1-t)^{-r-m-1}(\frac{X_1^{(n)}}{n}-t)^mdt],$
 which is the sum of two terms $P_1(m,n,s)$ and $P_2(m,n,s)$ with $0<s<1$ and

$$P_1(m,n,s)=\frac{\prod_{i=0}^{m}(r+i)}{m!}\mathbb{E}[\int_0^{\frac{X_1^{(n)}}{n}}(1-t)^{-r-m-1}(\frac{X_1^{(n)}}{n}-t)^mdt\ind_{X_1^{(n)}\geq  ns}],$$ 

$$P_2(m,n,s)=\frac{\prod_{i=0}^{m}(r+i)}{m!}\mathbb{E}[\int_0^{\frac{X_1^{(n)}}{n}}(1-t)^{-r-m-1}(\frac{X_1^{(n)}}{n}-t)^mdt\ind_{X_1^{(n)}<  ns }].$$

We first focus on $P_1(m,n,s)$. By Lemma \ref{equainfini} in Appendix C, we have
\begin{align}\label{p1}P_1(m,n,s) \leq \mathbb{E}[\left(\frac{n}{n-X_1^{(n)}}\right)^r\ind_{X_{1}^{(n)}\geq  ns }]\leq n^{-\alpha}K_4s^{-\alpha}(1-s)^{\bar{r}-r},\end{align}
where $\bar{r}\in (r,\omega)$ and $K_4$ is a number depending only on $\bar{r}$ and $\nu$(it is important to notice that it does not depend on $s$). 

We now give an upper bound for $P_2(m,n,s)$. We have

\begin{align*}
n^{\alpha}P_2(m,n,s)&=n^{\alpha}\frac{\prod_{i=0}^{m}(r+i)}{m!}\mathbb{E}[\int_0^{\frac{X_1^{(n)}}{n}}(1-t)^{-r-1}(\frac{X_1^{(n)}/n-t}{1-t})^m dt\ind_{X_1^{(n)}<  ns}].
\end{align*}
For $t\in[0,x)$ with $0<x\leq 1,$ we have:
$$\frac{x-t}{1-t}\leq x.$$

So we deduce that 
$$  \int_0^{\frac{X_1^{(n)}}{n}}(\frac{X_1^{(n)}/n-t}{1-t})^m dt\leq(\frac{X_1^{(n)}}{n})^{m+1}.$$

Hence, using Lemma \ref{xk}, for $m>2,$
\begin{align*}
n^{\alpha}P_2(m,n,s)&\leq n^{\alpha}\frac{\prod_{i=0}^{m}(r+i)}{m!}(1-s)^{-r-1}\mathbb{E}[(X_1^{(n)}/n)^{m+1}]\\
&=(1-s)^{-r-1}\frac{\prod_{i=0}^{m}(r+i)}{m!}(\frac{\int_{0}^1(m+1)t^{m}\rho(t)dt}{C_0\Gamma(2-\alpha)}+O(n^{-1})).
\end{align*}

{Using Lemme \ref{bound1} in Appendix C, we have }

\begin{align*}
\int_0^1(m+1)t^m\rho(t)dt
&=\int_0^1x^{m+1}\nu(dx)\\
&=\int_0^1x^{m+1}(1-x)^{\bar{r}}\frac{\nu(dx)}{(1-x)^{\bar{r}}}\\
&=\int_0^1x^{m+1}(1-x)^{\bar{r}}\nu^{(-\bar{r})}(dx)\leq K_5m^{-\bar{r}},
\end{align*}
where $K_5$ is a positive real number depending only on $\bar{r}$ and $\nu$.

Notice that  $\frac{\prod_{i=0}^{m}(r+i)}{m!}=\frac{\Gamma(m+r+1)}{\Gamma(r)\Gamma(m+1)}\sim \frac{m^{r}}{\Gamma(r)}$.
Hence 

\begin{equation}\label{p2}P_2(m,n,s)\leq n^{-\alpha}(1-s)^{-r-1}m^{r}(O(m^{-\bar{r}})+o(n^{-1})).\end{equation}

Combining $(\ref{p1})$ and $(\ref{p2})$, we deduce that 
$$\displaystyle \lim_{m\rightarrow +\infty}\limsup_{n\rightarrow +\infty}n^{\alpha}(P_1(m,n,s)+P_2(m,n,s))=\lim_{m\rightarrow +\infty}\limsup_{n\rightarrow +\infty}n^{\alpha}\frac{\prod_{i=0}^{m}(r+i)}{m!}\mathbb{E}[\int_0^{\frac{X_1^{(n)}}{n}}(1-t)^{-r-m-1}(\frac{X_1^{(n)}}{n}-t)^mdt]=0.$$
This convergence together with Lemma \ref{x1} and \ref{xk} yield this proposition.
\end{proof}

\begin{rem}
\label{rem_ext}
Using the same arguments, it is easy to prove that if $l\in {\mathbb N}$, $\mathbb{E}[(\frac{n}{n-X_1^{(n-l)}})^{r}]$ has also the decomposition given by Proposition~\ref{inverser}. More precisely, write 
\begin{align*}
\left(\frac{n}{n-X_1^{(n-l)}}\right)^{r}&=1+r\frac{X_1^{(n-l)}}{n}+\sum_{k=2}^{+\infty}\frac{\prod_{i=0}^{k-1}(r+i)}{k!}(\frac{X_1^{(n-l)}}{n})^k.
\end{align*}

For any $k\geq 1$, Lemma \ref{x1} and \ref{xk} give $\mathbb{E}[(\frac{X_1^{(n-l)}}{n})^k]=\mathbb{E}[(\frac{X_1^{(n-l)}}{n-l})^k]+O(n^{-2})$. This proposition also shows that 

$$n^{\alpha}\mathbb{E}[\sum_{k=2}^{+\infty}\frac{\prod_{i=0}^{k-1}(r+i)}{k!}(\frac{X_1^{(n-l)}}{n})^k]\leq n^{\alpha}\mathbb{E}[\sum_{k=2}^{+\infty}\frac{\prod_{i=0}^{k-1}(r+i)}{k!}(\frac{X_1^{(n-l)}}{n-l})^k]\stackrel{m\rightarrow +\infty}{\longrightarrow}0.$$

Then we can conclude.
\end{rem}
\begin{rem}\label{rem2}
If we move farther to Remark \ref{rem1}, we will see that  in the case of Beta$(2-\alpha, \alpha)$,  if $r\geq \omega$, we have
$$\lim_{m\rightarrow +\infty}\lim_{n\rightarrow +\infty}n^{\alpha}\frac{\prod_{i=0}^{m}(r+i)}{m!}\mathbb{E}[\int_0^{\frac{X_1^{(n)}}{n}}(1-t)^{-r-m-1}(\frac{X_1^{(n)}}{n}-t)^mdt]\geq C,$$
where $C>0$ is defined in Remark \ref{rem1}. Then in the third case of this Proposition \ref{inverser},  the expression of $\mathbb{E}[(\frac{n}{n-X_1^{(n)}})^r]$ will be changed. So the constraint $r\in [0,\omega)$ is necessary.
\end{rem}

C) Some results necessary to prove those in Appendix B
 { \begin{lem}\label{bin}
Let $B_{n,x}$ be a binomial random variable with parameter $(n,x),  n\geq 2, 0\leq x\leq 1$. Let $k$ be an integer such that $2\leq k\leq n$.
Then
$$nx+n(n-1)\cdots (n-k+1)x^k\leq \mathbb{E}[B_{n,x}^k]\leq (nx)^k+{k\choose 2}n^{k-1}x^2,$$
\end{lem}

\begin{proof}
Write $B_{n,x}=Y_1+\cdots +Y_n$, where  $Y_1,\cdots, Y_n$ are independent Bernoulli random variables. Let $S:=\{\{i_1,\cdots, i_k\}; 1\leq i_1,\cdots, i_k\leq n\}$. Then
\begin{align*}
\mathbb{E}[\sum_{\{i_1,\cdots, i_k\}\in S_1}Y_{i_1}\cdots Y_{i_k}]+\mathbb{E}[\sum_{\{i_1,\cdots, i_k\}\in S_3}Y_{i_1}\cdots Y_{i_k}]&\leq \mathbb{E}[(B_{n,x})^k]\\
&\leq  \mathbb{E}[\sum_{\{i_1,\cdots, i_k\}\in S_2}Y_{i_1}\cdots Y_{i_k}]+\mathbb{E}[\sum_{\{i_1,\cdots, i_k\}\in S_3}Y_{i_1}\cdots Y_{i_k}],
\end{align*}
where 
\begin{enumerate}
\item $S_1:=\{\{i_1,\cdots, i_n\}\in A; i_1=\cdots=i_k\}$. Then $\mathbb{E}[\sum_{\{i_1,\cdots, i_k\}\in S_1}Y_{i_1}\cdots Y_{i_k}]=nx$.
\item  $S_2:=\{\{i_1,\cdots, i_n\}\in A;  \exists 1\leq  p<q\leq k, i_p=i_q \}$. Then $\mathbb{E}[\sum_{\{i_1,\cdots, i_k\}\in S_2}Y_{i_1}\cdots Y_{i_k}]\leq {k\choose 2}n^{k-1}x^2$.
\item $S_3:=\{\{i_1,\cdots, i_n\}\in A; \forall 1\leq p<q\leq k,  i_p\neq i_q\}$. Then $\mathbb{E}[\sum_{\{i_1,\cdots, i_k\}\in S_3}Y_{i_1}\cdots Y_{i_k}]=n(n-1)\cdots(n-k+1)x^k$. 
\end{enumerate}

Then we can conclude.
\end{proof}
}

\begin{lem}\label{bound1}
We assume that $\rho(t)$ satisfies condition $(\ref{2c})$. Then for every $s\geq 0 $, 
$\int_0^1x^n(1-x)^{s}\nu(dx)\leq K_6n^{-s}$, where $K_6$ is a positive constant which depends only on $s$ and $\nu$.
\end{lem}

\begin{proof}
It is clear that there exists $K_7>0$ such that $\rho(t)\leq K_7t^{-\alpha}$, for all ${0<t\leq 1}$. Then
{\begin{align*}
\int_0^1x^n(1-x)^{s}\nu(dx)&=\int_0^1\rho(t)(n-(n+s)t)t^{n-1}(1-t)^{s-1}dt\\
&\leq \int_0^1\rho(t)(n-nt)t^{n-1}(1-t)^{s-1}dt\\
&\leq nK_7 \int_0^{1}t^{n-1-\alpha}(1-t)^sdt= nK_7\frac{\Gamma(n-\alpha)\Gamma(s+1)}{\Gamma(n-\alpha+s+1)}\leq K_6 n^{-s},
\end{align*}}
for some $K_6$ which only depends on $K_7$ and $s$. This achieves the proof of the lemma.

\end{proof}

The upper bound of $P_1(m,n,s)$ is given by the next lemma.
\begin{lem}\label{equainfini}
{We assume that $\rho(t)$ satisfies condition $(\ref{2c})$.  Let $r\in[0,\omega)$ and $\bar{r}\in(r,\omega)$. 
Then there exists a constant $K_{11}$ depending only on $\bar{r}$ and $\nu$ such that for all $s\in (0,1)$, 
}\begin{align*}P_1(m,n,s) \leq \mathbb{E}[\left(\frac{n}{n-X_1^{(n)}}\right)^r\ind_{X_{1}^{(n)}\geq  ns }]\leq n^{-\alpha}K_{11}s^{-\alpha}(1-s)^{\bar{r}-r}.\end{align*}
\end{lem}

\begin{proof} It is easy to observe that 
$$P_1(m,n,s) \leq \mathbb{E}[\left(\frac{n}{n-X_1^{(n)}}\right)^r\ind_{X_{1}^{(n)}\geq  ns }].$$
{For $x\in {\mathbb R}$, 
we define $\lceil x \rceil=\min\{m \in\mathbb{Z}; m\geq x\}$.} We have
\begin{align*}
\mathbb{E}[\left(\frac{n}{n-X_1^{(n)}}\right)^{{r}}\ind_{X_{1}^{(n)}\geq  ns }]
= \sum_{k= \lceil ns \rceil }^{n-1}\frac{\int_{0}^{1}{n\choose k+1}x^{k+1}(1-x)^{n-k-1}(\frac{n}{n-k})^{{r}}\nu(dx)}{g_{n}}.
\end{align*}
Notice that $ {n\choose k+1}(\frac{n}{n-k})^{r}=\frac{\Gamma(n+1)}{\Gamma(k+2)\Gamma(n-k)}(\frac{n}{n-k})^{r}.$ So by (\ref{ncd}), there exist two positive constants $K_8, K_9$ such that for all $ k\in \left\{1,2,\ldots,n-1\right\}$,  
$$K_8\frac{\Gamma(n+1+r)}{\Gamma(k+2)\Gamma(n-k+r)}\leq {n\choose k+1}(\frac{n}{n-k})^{r}\leq K_9\frac{\Gamma(n+1+r)}{\Gamma(k+2)\Gamma(n-k+r)}.$$
Moreover using integration by parts, for $1\leq l\leq n-1$ and $0\leq x\leq 1,$ we have:
\begin{align}\label{gamma}
&\sum_{k=l}^{n-1}\frac{\Gamma(n+1+r)}{\Gamma(k+2)\Gamma(n-k+r)}x^{k+1}(1-x)^{n-k-1+r}\nonumber\\
&=\frac{\Gamma(n+1+r)}{\Gamma(l+1)\Gamma(n-l+r)}\int_{0}^{x}t^{l}(1-t)^{n-l+r-1}dt+\frac{\Gamma(n+1+r)}{\Gamma(n+1)\Gamma(1+r)}x^{n}(1-x)^{r}\\
&-\frac{\Gamma(n+1+r)}{\Gamma(n)\Gamma(1+r)}\int_0^x t^{n-1}(1-t)^{r}dt.\nonumber
\end{align}
Let $\bar{r}\in(r,\omega)$, $\nu^{(-\bar{r})}(dx)=\frac{\nu(dx)}{(1-x)^{\bar{r}}}$, and $\rho^{(-\bar{r})}(t)=\int_{t}^1\nu^{(-\bar{r})}(r)dr$. It is easy to see that $\rho^{(-\bar{r})}(t)=C_0t^{-\alpha}+o(t^{-\alpha}).$  Then there exists $K_{10}>0$, such that $\rho^{(-\bar{r})}(t)\leq K_{10}t^{-\alpha}$ for all $t\in(0,1].$

\begin{align}
&\mathbb{E}[\left(\frac{n}{n-X_1^{(n)}}\right)^{\bar{r}}\ind_{X_{1}^{(n)}\geq  ns }]\nonumber\\
&
= \sum_{k= \lceil ns \rceil }^{n-1}\frac{\int_{0}^{1}{n\choose k+1}x^{k+1}(1-x)^{n-k-1}(\frac{n}{n-k})^{\bar{r}}\nu(dx)}{g_{n}}\nonumber= \sum_{k=\lceil ns \rceil}^{n-1}\frac{\int_{0}^{1}{n\choose k+1}x^{k+1}(1-x)^{n-k-1+\bar{r}}(\frac{n}{n-k})^{\bar{r}}\nu^{(-\bar{r})}(dx)}{g_{n}}\nonumber\\
&\leq K_9\frac{\int_0^1\frac{\Gamma(n+1+\bar{r})}{\Gamma(\lceil ns \rceil+1)\Gamma(n-\lceil ns \rceil+\bar{r})}\int_0^xt^{\lceil ns \rceil}(1-t)^{n-\lceil ns \rceil+\bar{r}-1}dt\nu^{(-\bar{r})}(dx)}{g_n}+K_9\frac{\int_0^1\frac{\Gamma(n+1+\bar{r})}{\Gamma(n+1)\Gamma(1+\bar{r})}x^n(1-x)^{\bar{r}}\nu^{(-\bar{r})}(dx)}{g_n}\nonumber\\
&\leq K_9\frac{\int_0^1\frac{\Gamma(n+1+\bar{r})}{\Gamma( \lceil ns \rceil+1 )\Gamma(n- \lceil ns \rceil +\bar{r})}\rho^{(-\bar{r})}(x)x^{ \lceil ns \rceil }(1-x)^{n-\lceil ns \rceil+\bar{r}-1}dx}{g_n}+K_9\frac{\int_0^1\frac{\Gamma(n+1+\bar{r})}{\Gamma(n+1)\Gamma(1+\bar{r})}x^n(1-x)^{\bar{r}}\nu^{(-\bar{r})}(dx)}{g_n}\nonumber\\
&\leq K_9K_{10}\frac{\frac{\Gamma(n+1+\bar{r})\Gamma(\lceil ns \rceil+1-\alpha)}{\Gamma(\lceil ns \rceil+1)\Gamma(n+1+\bar{r}-\alpha)}}{g_n}+{K_6}K_9\frac{\frac{\Gamma(n+1+\bar{r})}{\Gamma(n+1)\Gamma(1+\bar{r})}n^{-\bar{r}}}{g_n}\nonumber\\
&\leq K_{11}s^{-\alpha}n^{-\alpha},\nonumber
\end{align}

where for the first inequality, we use (\ref{gamma}) with $l=\lceil ns \rceil$, in the second inequality, we have used an argument of integration by parts and for the third inequality, we bound $\rho^{(-\bar{r})}(x)$ by $K_{10}x^{-\alpha}$ and we also use Lemma \ref{bound1}. For the last inequality, we use (\ref{ncd}).
 {Here $K_{11}$ is a constant which depends only  on $\bar{r}$ and $\nu$.} 
Hence for all $n\geq 2$, 
\begin{align*}n^{\alpha}\mathbb{E}[\left(\frac{n}{n-X_1^{(n)}}\right)^r\ind_{X_{1}^{(n)}\geq  ns }]\leq n^{\alpha}(1-s)^{\bar{r}-r}\mathbb{E}[\left(\frac{n}{n-X_1^{(n)}}\right)^{\bar{r}}\ind_{X_{1}^{(n)}\geq  ns }]\leq K_{11}s^{-\alpha}(1-s)^{\bar{r}-r}.\end{align*}
This achieves the proof of the lemma.
\end{proof}
{\begin{rem}\label{rem1}
If $r\geq \omega,$ this lemma is false. Take Beta$(2-\alpha, \alpha)$ as example. We have $\nu(dx)=\frac{1}{\Gamma(\alpha)\Gamma(2-\alpha)}x^{-1-\alpha}(1-x)^{\alpha-1}dx$ and $\omega=\alpha.$ Then for any fixed $0<s<1$ and $n\geq \frac{1}{1-s}$, we have $ns\geq n-1$ and it follows that 
\begin{align*}
& P_1(m,n,s)\\
&\geq \mathbb{E}[\left(\left(\frac{n}{n-X_1^{(n)}}\right)^{r}-1-r\frac{X_1^{(n)}}{n}-\sum_{k=2}^{m}\frac{\prod_{i=0}^{k-1}(r+i)}{k!}(\frac{X_1^{(n)}}{n})^k\right)\ind_{X_1^{(n)}=n-1}]\\
&=\mathbb{P}(X_1^{(n)}=n-1)\left(n^{r}-1-r\frac{n-1}{n}-\sum_{k=2}^{m}\frac{\prod_{i=0}^{k-1}(r+i)}{k!}(\frac{n-1}{n})^k\right)\\
&=\frac{\int_0^1x^n\nu(dx)}{g_n}\left(n^{r}-1-r\frac{n-1}{n}-\sum_{k=2}^{m}\frac{\prod_{i=0}^{k-1}(r+i)}{k!}(\frac{n-1}{n})^k\right)\\
&\sim Cn^{-2\alpha}\left(n^{r}-1-r\frac{n-1}{n}-\sum_{k=2}^{m}\frac{\prod_{i=0}^{k-1}(r+i)}{k!}(\frac{n-1}{n})^k\right)\\
\end{align*}
where $C$ is a strictly positive number. Hence, if $r\geq \omega=\alpha$, then 
$$\displaystyle \liminf_{n\rightarrow +\infty}n^{\alpha} P_1(m,n,s)\geq C, \forall 0<s<1.$$
This result is not compatible with the lemma with $0\leq r<\omega$. This remark justifies the constraint  $0\leq r<\omega$.
\end{rem}
}

D) Results that are used to prove Theorem \ref{secondtheo}.

\begin{lem}\label{inequality}
Let $a> 0,b> 0$, $\beta>2$. Then 
\[
0< (a+b)^{\beta}\leq a^{\beta}+b^{\beta}+\beta2^{\beta-1}ab^{\beta-1}+\beta2^{\beta-1}ba^{\beta-1}.
\]
\end{lem}
\begin{proof}
If $0\leq m \leq 1$, then 
 $$(1+m)^{\beta}\leq 1+\beta2^{\beta-1}m\leq 1+m^{\beta}+\beta2^{\beta-1}m+\beta2^{\beta-1}m^{\beta-1}.$$
 We use that the function $m\mapsto (1+m)^{\beta}$ is convex and that  $\beta2^{\beta-1}$ is the derivative of $(1+m)^{\beta}$ at $m=1$.

If $1<m$, then
$$(1+m)^{\beta}=m^{\beta}(1+\frac{1}{m})^{\beta}\leq (m)^{\beta}(1+\beta2^{\beta-1}\frac{1}{m})\leq 1+m^{\beta}+\beta2^{\beta-1}m+\beta2^{\beta-1}m^{\beta-1}.$$
Hence for all $m>0$,
$$(1+m)^{\beta}\leq 1+m^{\beta}+\beta2^{\beta-1}m+\beta2^{\beta-1}m^{\beta-1}.$$
Then for all $a>0,b>0$,
\begin{align*}
(a+b)^{\beta}=a^{\beta}(1+\frac{b}{a})^{\beta}&\leq a^{\beta}(1+(\frac{b}{a})^{\beta}+\beta2^{\beta-1}\frac{b}{a}+\beta2^{\beta-1}(\frac{b}{a})^{\beta-1})\\
&
=a^{\beta}+b^{\beta}+\beta2^{\beta-1}ab^{\beta-1}+\beta2^{\beta-1}ba^{\beta-1}.\\
\end{align*}
This achieves the proof.
\end{proof}

\bibliographystyle{abbrv}
\bibliography{yll}

\end{document}